\documentclass[a4paper,10pt,reqno]{amsart}
\usepackage{amssymb,bbm}
\usepackage{array}
\usepackage{mathtools}
\usepackage{graphicx}
\newcolumntype{L}{>{\centering\arraybackslash}m{1.5cm}}
\usepackage{amsthm, calrsfs}
\usepackage[mathscr]{eucal}
\usepackage{rotating}
\usepackage{multirow,float}
\usepackage{color}
\numberwithin{equation}{section}
\newtheorem{theorem}{Theorem}
\newtheorem{lemma}{Lemma}
\newtheorem{proposition}{Proposition}
\newtheorem{corollary}{Corollary}

\theoremstyle{definition}
\newtheorem{definition}[theorem]{Definition}

\newtheorem{examples}[theorem]{Example}
\theoremstyle{remark}
\usepackage[colorlinks=true, citecolor=blue]{hyperref}

\usepackage{enumerate}

\newcommand{\RN}[1]{\uppercase\expandafter{\romannumeral#1}}
\newcommand{\Rn}[1]{\romannumeral#1\relax}

\begin{document}
\title[Growth Rates of Sublinear FDEs]{Growth Rates of Sublinear Functional and Volterra Differential Equations}

\author{John A. D. Appleby}
\address{School of Mathematical
Sciences, Dublin City University, Glasnevin, Dublin 9, Ireland}
\email{john.appleby@dcu.ie} \urladdr{webpages.dcu.ie/\textasciitilde applebyj}

\author{Denis D. Patterson}
\address{School of Mathematical
Sciences, Dublin City University, Glasnevin, Dublin 9, Ireland}
\email{denis.patterson2@mail.dcu.ie} \urladdr{sites.google.com/a/mail.dcu.ie/denis-patterson}

\thanks{Denis Patterson is supported by the Government of Ireland Postgraduate Scholarship Scheme operated by the Irish Research Council under the project GOIPG/2013/402.} 
\date{\today}

\begin{abstract}
This paper considers the growth rates of positive solutions of scalar nonlinear functional and Volterra differential equations. 
The equations are assumed to be autonomous (or asymptotically so), and the nonlinear dependence grows less rapidly than any linear function. We impose extra regularity properties on a function asymptotic to this nonlinear function, rather than on the nonlinearity itself. The main result of the paper demonstrates that the growth rate of the solution can be found by determining the rate of growth of a trivial functional differential equation (FDE) with the same nonlinearity and all its associated measure concentrated at zero; the trivial FDE is nothing other than an autonomous nonlinear ODE. We also supply direct asymptotic information about the solution of the FDE under additional conditions on the nonlinearity, and exploit the theory of regular variation to sharpen and extend the results. 
\end{abstract}
\maketitle

\section{Introduction}
We study the asymptotic behaviour of unbounded solutions of the following (asymptotically autonomous convolution) functional differential equation
\begin{align}\label{functional}
x'(t) &= \int_{[0,\tau]}\mu_1(ds)f(x(t-s)) + \int_{[0,t]} \mu_2(ds)f(x(t-s)), \quad t>0;\\ \nonumber
 x(t) &= \psi(t), \quad t \in [-\tau,0], \quad \tau \in (0,\infty),
\end{align}
where $\mu_1, \mu_2$ are positive finite measures, $f$ is a positive continuous nonlinear function, and $\psi$ is a positive continuous function. By the Riesz representation theorem, \eqref{functional} is equivalent to $x'(t)=L([f(x)]_t)$, $t>0$ where $L$ is a positive continuous linear functional from $C([-\tau,\infty);\mathbb{R}^+)$ to $\mathbb{R}^+$.
Growth estimates on the solutions of such nonlinear convolution--type equations have attracted much attention 
(see e.g., Lipovan \cite{lipovan2008,lipovan2009}, and Schneider \cite{schneider1982}).

We will assume, with a view to applications in mathematical economics, that $f$ is sublinear in the sense that $f(x)/x\to 0$ as $x\to\infty$. Under these conditions, solutions of \eqref{functional} will grow but will not exhibit finite time blow up; more precisely $x \in C([-\tau,\infty);(0,\infty))$ but $\lim_{t\to\infty}x(t)=\infty$. We may then ask whether the asymptotic growth rates of solutions to \eqref{functional} can be captured in a meaningful way. Our main results provide sufficient conditions under which the solutions of \eqref{functional} have essentially the same asymptotic behaviour as the related autonomous ordinary differential equation
\begin{equation}\label{ODE}
y'(t) = M f(y(t)), \quad t>0;\quad y(0)= \psi, \quad M:=\int_{[0,\tau]} \mu_1(ds) + \int_{[0,\infty)} \mu_2(ds).
\end{equation}
Indeed, defining $F$ by 
\[
F(x)=\int_1^x \frac{1}{f(u)}\,du, \quad x >0,
\]
the sublinearity of $f$ implies that $F(x)\to\infty$ and our most general results show that 
\[
\lim_{t\to\infty} \frac{F(x(t))}{Mt}=1,
\]
which, under strengthened conditions can sometimes be improved to give direct asymptotic information in the form
\[
\lim_{t\to\infty} \frac{x(t)}{F^{-1}(Mt)}=1.
\] 
Since  
\[
y(t)=F^{-1}(Mt+F(\psi)), \quad t\geq 0,
\]
we automatically have that $\lim_{t\to\infty} F(y(t))/t=M$. Also
\[
 \lim_{t\to\infty} \frac{y(t)}{F^{-1}(Mt)}=1.
\]
This is because the sublinearity of $f$ implies $y'(t)/y(t)\to 0$ as $t\to\infty$: this in turn implies that $y(Mt+c)/y(Mt)\to 1$ as $t\to\infty$ for any $c\in\mathbb{R}$, which implies the claim by choosing $c=-F(\psi)$. Therefore, the asymptotic results for $x$ mirror those for $y$ precisely. 

In order to determine explicit first order representations for the asymptotic behaviour of $y(t)$ as $t\to\infty$, it is only necessary to determine the large time behaviour of $F$ and $F^{-1}$, and this is precisely what is needed to determine explicit first order representations for the asymptotic behaviour of $x(t)$ as $t\to\infty$. Therefore, in this sense, finding the asymptotic behaviour of \eqref{functional} \emph{reduces exactly to the related problem for the ODE}.    

The asymptotic theory of equations such as \eqref{functional} is intimately related to the upper bound estimates furnished by inequalities of the Bellman-Bihari-Gronwall type (cf.~\cite{bihari1956generalization}) and in some sense this paper addresses the question of when these estimates are asymptotically sharp. The literature on such inequalities is vast, and important results are given in several monographs (e.g., Lakshmikantham and Leela \cite{lakleela12}, and Pachpatte \cite{ames1997inequalities}). In \cite{ames1997inequalities}, Pachpatte provides myriad examples of differential inequalities and their applications to the qualitative theory of differential equations, including to linear integro-differential equations similar to \eqref{functional}. However, we note that these results rely heavily on the nonlinear function being monotonically increasing and generally only provide upper estimates on the size of solutions. Indeed much of the literature on growth bounds and estimates involves monotone hypotheses; these are of course natural when trying to establish uniqueness of solutions of dynamical systems since the nonlinearity can be interpreted as a modulus of continuity, which enjoys natural monotonicity properties. 

In fact, if we assume that $f$ is \emph{increasing}, and that 
$F(x)\to\infty$ as $x\to\infty$, immediate integration of \eqref{functional} leads directly to an \emph{upper} inequality of Bihari type, and therefrom an estimate of the form 
\[
F(x(t))\leq C+Mt, \quad t\geq 0
\]
for some $C>0$. This immediately yields 
\[
\limsup_{t\to\infty} \frac{F(x(t))}{Mt}\leq 1.
\]
However, this result does not indicate whether the estimate is in any sense sharp for non--trivial functional differential equations (although it is certainly so for sublinear ODEs). It is well--known that the estimate from the Bihari--Gronwall--Bellman approach cannot be sharp when $f$ is linear, as exact Liapunov exponents -- which do not coincide with those resulting from the Gronwall inequality -- are given by solving the characteristic equation (see e.g., \cite{gripenberg1990volterra}). Therefore, it is of evident interest to develop corresponding \emph{lower} inequalities on the solution of \eqref{functional} with a view to investigating the quality of the upper bounds generated by the standard theory. An excellent paper which addresses such lower bounds in Volterra integral equations is that of Lipovan~\cite{lipovan2006}; in the present work we develop suitable lower inequalities for the solutions of our equations which in fact highlight the sharp character of the bounds achieved in \cite{lipovan2006}. We are assisted in this task by the integro--differential character of \eqref{functional}. Furthermore, the standard approach does not seem to address the situation in which $f$ is \emph{decreasing}. To a certain degree, the main new contributions of this work are to furnish \emph{sharp} lower estimates, to relax the hypothesis that $f$ increases, and to obtain simplified but precise limiting behaviour, rather than explicit and global bounds on the solution. Indeed, as we are in any case studying differential systems, we find it sometimes reasonable not to integrate \eqref{functional}, in part to prevent the destruction of useful information about the solution, and this leads to a different line of attack from the aforementioned integral equation theory.

In fact, in this work we obtain exact asymptotic estimates of the growth rates of solutions to \eqref{functional}. In the linear case, as noted above, the exact asymptotic behaviour is known, and the upper estimate is not sharp. If $f$ grows sufficiently more rapidly than linearly, finite time blow--up of solutions is possible. Therefore, we confine our attention to the case when $f$ is sublinear.   A characterisation of sublinearity which seems mild is that $f$ is asymptotic to a function which has zero derivative as we approach infinity. Sublinear equations of this type were studied by Appleby et al. \cite{subexpRV} with a single delay term but this analysis relies on the theory of regular variation. Other works which give growth estimates for sublinear functional differential equations include Graef~\cite{graef}, and Kusano and Onose~\cite{KO}. In this article we impose the more general hypothesis of asymptotic monotonicity on the nonlinear function $f$; both the increasing and decreasing cases are addressed. This generality allows us to easily recover the results for the case of regular variation. However, our analysis reveals that relaxing either the increasing or decreasing hypothesis completely makes estimation of a sharp growth bound difficult. In the case when $f$ is slowly varying at infinity we can still achieve results but the unbounded delay case is challenging. Only under additional hypotheses on $f$ can we obtain exact asymptotics in this case. 

A primary motivation for this work is to give a platform to deduce growth and fluctuation properties for deterministically and stochastically perturbed functional and Volterra equations of the form 
\begin{align} \label{eq.stochperturbed}
dX(t)&= \left(\int_{[0,\tau]}\mu_1(ds)f(X(t-s)) + \int_{[0,t]} \mu_2(ds)f(X(t-s)) + h(t)\right)\,dt \\ 
&\qquad\qquad+ dZ(t), \quad t\geq 0, \nonumber\\
X(t)&=\psi(t), \quad t\in [-\tau,0], \nonumber
\end{align}
(where $Z$ is a semimartingale with appropriate asymptotic properties). Systems with the same qualitative features as those present in \eqref{eq.stochperturbed} find applications in the endogenous growth theory of mathematical economics and in particular in vintage capital models.

The inclusion of general finite measures in \eqref{eq.stochperturbed} is a key feature for applications to vintage capital as it allows both demographic and structural delay effects to be captured; the work of Benhabib and Rustichini \cite{benhabib1991vintage} is an excellent early exemplar of how Volterra equations with general measures can be used to model non--exponential depreciation of capital, and effects such as ``learning by doing'' and time--to--build lags. These ideas, and variants thereof, have subsequently been developed in both the economic and mathematical literature. d'Albis et al. \cite{d2014multiple} (and the references contained therein) provide a more up to date overview of the development of such models and the associated mathematical machinery. We note that the aforementioned literature is primarily focused on models involving equations which are linear in the state variable.

The second notable qualitative feature of equation \eqref{eq.stochperturbed} is the inclusion of a sublinear nonlinearity; this arises naturally in economic models as a consequence of the so--called law of diminishing returns. In the context of endogenous growth models, Jones \cite{jones1995r} explains the crucial need to incorporate non--unit returns to scale in order to eliminate unrealistic scale effects (see also \cite{jones1995time}). The most common sublinearities in the economic literature are those of power type and our treatment of the case of regular variation is therefore especially pertinent in this context. The work of Lin and Shampine \cite{lin2014finite} provides a recent example of the intersection of endogenous growth theory and the theory of functional differential equations. Building on the framework of Jones and Williams \cite{jones2000too}, Lin and Shampine present a model of finite length patents in a decentralised economy which gives rise to a complex system of functional differential equations with sublinear state--dependent terms.    

Finally, the inclusion of a deterministic state--independent term $h$ in \eqref{eq.stochperturbed} serves to model underlying trends in the external environment, while the semimartingale term models exogenous and uncertain pertubations to the system. Interesting examples of appropriate semimartingales include 
\[
Z_1(t)=\int_0^t \sigma(s)\,dB(s), \quad Z_2(t)=Y_\alpha(t)
\] 
where $B$ is standard Brownian motion, and $Y_\alpha$ is an $\alpha$--stable L\'evy process (see Bertoin \cite{bertoin1998levy} for further details). These allow us to  model respectively systems subject to exogeneous, persistent time--dependent shocks, or systems perturbed by erratic, and potentially large, shocks. The state--independence of $Z$ in our examples make \eqref{eq.stochperturbed} reminiscent of a continuous--time nonlinear time series model subject to white noise perturbations (cf. Brockwell and Lindner \cite{brockwell2009existence}, or Marquardt and Stelzer \cite{marquardt2007multivariate}). 

Given that the asymptotic properties of \eqref{eq.stochperturbed} are surely conditioned by the relative sizes of $h$, the large deviations in $Z$, and the growth in the unperturbed system arising from $f$ and $\mu_{1,2}$, it is obvious that a thorough understanding of the asymptotic behaviour in the unperturbed case is an essential ingredient in developing good asymptotic results for solutions of \eqref{eq.stochperturbed}. This is especially true for the important question as to what size of the stochastic shocks will fundamentally change the growth rate of the shock--free system. Furthermore, it is hard to conceive of a good asymptotic theory for systems with state--dependence stochastic forcing (such as 
\begin{multline*}
dX(t)= \left(\int_{[0,\tau]}\mu_1(ds)f(X(t-s)) + \int_{[0,t]} \mu_2(ds)f(X(t-s)) + h(t)\right)\,dt \\ 
+ G(t,X_t)\,dZ(t)
\end{multline*}
where $Z$ is a semimartingale of the type mentioned earlier, and $G$ is a functional with appropriate asymptotic and 
regularity properties) without first understanding the dynamics of \eqref{functional} and  \eqref{eq.stochperturbed}. 

We also note that for such stochastic systems, which are likely also subject to model uncertainty, there is less potential for global pathwise bounds to be of value. This further supports our emphasis on asymptotic results with less stringent requirements on $f$.  

The organisation of the paper is as follows: the next section introduces notation and discusses hypotheses imposed on the nonlinearity in order to guarantee existence, non--explosion and uniqueness of solutions, as well as to allow the development of comprehensive 
asymptotic results. Section \ref{general_results_section} discusses the asymptotic results that can be established by imposing the general conditions on $f$ discussed in the previous section; Section \ref{sec.RV} presents results in the important case where $f$ is a regularly varying function at infinity. Section \ref{sec.examples} outlines some examples of nonlinear functions $f$ which violate monotonicity hypotheses, but are allowed by our results. In general, the proofs of our main results are postponed to the final sections of the paper. Section \ref{sec.mainresults} deals with results where $f$ is asymptotically increasing, Section \ref{sec.decreasing.results} with the case where $f$ is asymptotically 
decreasing, and Section \ref{sec.regvarproofs} where $f$ is regularly varying. The final section contains some computations claimed in the presentation of the examples in Section \ref{sec.examples}. 
 
\section{Mathematical Preliminaries and Discussion of Hypotheses}\label{discussion}
Throughout this paper we employ the notational convention $\mathbb{R}^+:=[0,\infty)$.

We begin by defining a useful equivalence relation on the space of positive continuous functions; in essence, we consider two functions to be equivalent if they have the same leading order asymptotic behaviour.
\begin{definition}
Suppose $b,c \in C(\mathbb{R}^+;\mathbb{R}^+)$. $b$ and $c$ are said to be asymptotically equivalent if 
$\lim_{t\to\infty}b(t)/c(t)=1$. We write $b(t) \sim c(t)$ as $t\to\infty$, or sometimes $b\sim c$ for extra brevity.
\end{definition}
Occasionally, we will employ the standard Landau ``O'' and ``o'' notation. If $c$ is as above and $b\in C(\mathbb{R}^+;\mathbb{R})$, we write $b(t)=O(c(t))$ if $|b(t)|\leq Kc(t)$ for some $K\in (0,\infty)$ and $t$ sufficiently large, and $b(t)=o(c(t))$ if 
$b(t)/c(t)\to 0$ as $t\to\infty$. 

For equation \eqref{functional}, we will concentrate on the case when $\mu_1$ and $\mu_2$ are non--negative, finite Borel measures on $(\mathbb{R}^+,\mathcal{B}(\mathbb{R}^+))$; more precisely
\begin{align}\label{finite_measure}
\mu_i(E)\geq 0,\, E\in\mathcal{B}(\mathbb{R}^+), \, i = 1,2; \,\, \int_{[0,\tau]} \mu_1(ds) + \int_{[0,\infty)} \mu_2(ds) =: M \in \mathbb{R}^+.
\end{align}
Of course, in order to avoid a trivial right--hand side of \eqref{functional}, at least one of $\mu_1$ and $\mu_2$ should be 
positive measures. 

Suppose that 
\begin{align}\label{global_stable}
f \in C(\mathbb{R}^+;(0,\infty)).
\end{align}
Then the function $F:(0,\infty)\to \mathbb{R}$ by 
\begin{equation}\label{cap_F}
F(x) = \int_1^x \frac{1}{f(u)}du, \quad x>0,
\end{equation}
is well--defined and \eqref{global_stable} ensures that $F$ is strictly increasing and hence invertible. 
Throughout this paper, whenever $\mu_1$ is non--trivial, we will assume that the initial function $\psi$ is strictly positive on the initial interval since this allows us to guarantee that the solution $x$ of \eqref{functional} obeys $x(t)\to+\infty$ as 
$t \to \infty$. However, we note that this assumption does not enter into any of our arguments regarding the rate of growth and hence we can apply these arguments if it is known independently that the solution grows to $+\infty$. 

Our work concerns sublinear equations, and we pause here to describe what \emph{we} mean by sublinear in this context, as this terminology has precise meaning in other settings. We say that $f$ is \textit{sublinear} if it is dominated by every positive linear function at infinity. This means that $f$ obeys 
\begin{equation} \label{eq.fsublinear}
\lim_{x\to\infty} \frac{f(x)}{x}=0.
\end{equation}
We presently impose an additional condition on $f$ which implies sublinearity in this sense. 
We note that the continuity of $f$, together with sublinearity, imply that $f$ satisfies a global linear growth bound of the form 
\[
\text{There exists $K>0$ such that } 0<f(x)\leq K(1+x), \quad x>0.
\]
Such a condition is well--known to be sufficient to exclude the possibility that the solutions to \eqref{functional} blow up in finite time. With the possibility of blow-up excluded we note that \eqref{functional} will have a unique continuous solution if we additionally suppose that $f$ is locally Lipschitz continuous. The existence and uniqueness theory for the equations studied in this article is well-known and the interested reader can find a thorough exposition of this material in the classic text of Gripenberg, Londen and Staffans  \cite{gripenberg1990volterra}.

The non-standard mixed form of \eqref{functional} owes to the fact that the methods of this paper can be equally well applied to both bounded delay equations and Volterra equations without a forcing term on the right--hand side; for brevity we prove results for \eqref{functional} which can be immediately applied to each special class of equations as desired. We also remark that we could rewrite \eqref{functional} as a ``pure'' Volterra equation at the expense of adding an exogenous forcing term by noting that delay differential equations of the form
\begin{align}\label{DDE}
z'(t) = \int_{[0,\tau]}\mu_1(ds)f(z(t-s)),\,\,t \geq 0;\,\, z(t) = \psi(t), \,\, t \in [-\tau,0],
\end{align}
can be written as 
\begin{align}\label{Volterra_general}
z'(t) = \int_{[0,t]}\mu(ds)f(z(t-s))+ h(t), \,\, t\geq 0;\,\, z(0) = \psi(0),
\end{align}
where $\mu(E) = \mu_1(E \cap [0,\tau])$ and $h(t)= \mathbbm{1}_{[0,\tau)}(t) \int_{[t,\tau]}\mu_1(ds)f(\psi(t-s))$. Indeed our results for \eqref{functional} can be viewed as a necessary first step in understanding the \emph{growth} asymptotics of unbounded solutions of more general functional equations of the form \eqref{Volterra_general}, for an arbitrary forcing term $h$. Exact rates of \emph{decay} for forced ordinary nonlinear stochastic equations and deterministic equations, in which the state dependence is weaker than linear at equilibrium, and the perturbation can be large enough to give rise to new asymptotic behaviour, are given by Appleby and Patterson \cite{appleby2013classification,appleby2014necessary}.

In order to prove many of our results, we request some extra properties and regularity on $f$ that are not necessarily satisfied by sublinear functions as described by \eqref{eq.fsublinear}. However, we  believe that our choice of additional hypotheses on $f$ are not especially restrictive, relatively natural in the context of differential systems, and apply in a unified manner across a variety of situations. In particular, although we do not discuss the case of forced equations or stochastic functional differential equations in this paper, our additional hypotheses on $f$ are chosen with a view that good asymptotic results can still be obtained for such perturbed systems without the need for extra assumptions on the nonlinearity, even in the case when the forcing term may be very large. Finally, our hypothesis allows us to give a good characterisation of the asymptotic behaviour  without recourse to especially elaborate proofs. 
In order to specify our hypothesis on $f$ we first introduce the class of functions
\begin{multline} \label{def.smoothsublinear}
\mathcal{S}=\{\phi\in C^1((0,\infty);(0,\infty))\cap C(\mathbb{R}^+,(0,\infty)): 
\text{$\lim_{x\to\infty}\phi'(x)=0$} \\
\text{and $\phi'(x) > 0$ for all $x>0$} \,\,\}.
\end{multline}
We now make our assumption on $f$: 
\begin{equation} \label{fasym}
\text{With $\mathcal{S}$ as in \eqref{def.smoothsublinear}, there is $\phi\in \mathcal{S}$ such that $f(x) \sim \phi(x)$ as $x\to\infty$}.
\end{equation}
We see immediately (cf. Lemma~\ref{to_zero}) that \eqref{fasym} implies the sublinear property \eqref{eq.fsublinear}.

One of the principal advantages of the strengthened hypothesis \eqref{fasym} is that the extra regularity requirements, such as monotonicity and smoothness, are \emph{not imposed directly} on $f$ but rather on the auxiliary function $\phi$. This allows $f$ to have a certain irregularity without any cost. 

While \eqref{fasym} holds for large classes of sublinear functions that are commonly found in applications, it is still a strictly stronger hypothesis than sublinearity, even when $f$ is increasing. In particular, an increasing, continuously differentiable, sublinear function $f$ must have $\liminf_{x\to\infty}f'(x)=0$ but there is no guarantee that $\limsup_{x\to\infty}f'(x)=0$ and it is even possible to have 
\begin{equation} \label{eq.f'0infty}
0 = \liminf_{x\to\infty}f'(x) < \limsup_{x\to\infty}f'(x)=\infty.
\end{equation}
Indeed, our asymptotic results can still be applied to functions $f$ obeying \eqref{eq.f'0infty}, provided intervals on which $f$ has bad behaviour are relatively short. We illustrate this point more fully in Section \ref{sec.examples} with some examples.   
Properties of the derivative of $\phi$ in \eqref{fasym} allow simple application of the mean--value theorem at crucial junctures in our proofs, and bypass the potential difficulties posed by functions $f$ with pathological behaviour of the type exemplified by  
\eqref{eq.f'0infty}.

In the next section, we concentrate on equations in which $f$ obeys general sublinearity hypotheses (such as \eqref{fasym}). However, an important class of functions which overlaps with positive sublinear functions is the class of regularly varying functions, and by making the extra assumption of regular variation on $f$, we will see in Section~\ref{sec.RV} that stronger conclusions concerning asymptotic behaviour can be made. Since we mention this class in advance of Section~\ref{sec.RV}, we recall here the definition of a regularly varying function for the reader's convenience.
\begin{definition}
Suppose a measurable function $h:\mathbb{R}\mapsto (0,\infty)$ obeys
\[
\lim_{x\to\infty}\frac{h(\lambda\,x)}{h(x)} = \lambda^\beta, \mbox{ for all } \lambda>0, \mbox{ some }\beta\in\mathbb{R},
\]
then $f$ is regularly varying at infinity with index $\beta$, or $f\in \text{RV}_\infty(\beta)$.
\end{definition}
Further properties of regularly varying functions will be discussed and quoted in Section~\ref{sec.RV}.  

\section{General Results}\label{general_results_section}
We start by remarking that sublinear behaviour in $f$ implies subexponential growth in the solution $x$ of \eqref{functional}, in the sense that $x(t)$ has an instantaneous growth rate that tends to zero as $t\to\infty$, or that $x$ has a zero Liapunov exponent. The proof is elementary, and we give it immediately below.
\begin{theorem}  \label{thm.sub}
Suppose that the measures $\mu_1$ and $\mu_2$ obey \eqref{finite_measure}, $f$ obeys \eqref{global_stable}, 
\eqref{eq.fsublinear}, and $\psi\in C([-\tau,0];(0,\infty))$. Then, the unique continuous solution, $x$, of \eqref{functional} obeys $x(t)\to\infty$ as $t\to\infty$ and moreover
\begin{equation} \label{eq.thmsub1}
\lim_{t\to\infty} \frac{x'(t)}{x(t)}=0, \quad \lim_{t\to\infty} \frac{1}{t}\log x(t)=0.
\end{equation}
\end{theorem}
\begin{proof}
It is shown later in the paper that \eqref{finite_measure}, \eqref{global_stable}, and the positivity of $\psi$ are sufficient 
to show that $x(t)\to\infty$ as $t\to\infty$ and indeed that $x'(t)\geq 0$ for all $t>0$. Since $f$ 
is continuous and obeys  \eqref{eq.fsublinear}, for every $\epsilon>0$ there is $L(\epsilon)>0$ such that 
$0<f(x)<L(\epsilon)+\epsilon x$ for all $x\geq 0$. Hence for $t>\tau$ we have that $x(t+s)\leq x(t)$ for all $-\tau\leq s\leq 0$, and so 
\begin{align*}
0<x'(t)
&\leq \int_{[-\tau,0]} \mu_1(ds) \left\{L(\epsilon)+\epsilon x(t+s)\right\} + \int_{[0,t]} 
\mu_2(ds)\left\{L(\epsilon)+\epsilon x(t-s)\right\} \\
& \leq L(\epsilon) M + \epsilon M x(t),
\end{align*}
where $M=\mu_1([-\tau,0])+\mu_2(\mathbb{R}^+)$. Since $x(t)\to\infty$ as $t\to\infty$, we have that $\limsup_{t\to\infty} x'(t)/x(t)\leq \epsilon M$. Since $\epsilon>0$ is arbitrary, it follows immediately that $\limsup_{t\to\infty} x'(t)/x(t)\leq 0$. Since $x$ is increasing, we 
 have $\liminf_{t\to\infty} x'(t)/x(t)\geq 0$, which proves the first part of \eqref{eq.thmsub1}. From the first part of \eqref{eq.thmsub1}, and monotonicity of $x$ we have for every $\epsilon>0$ that there is $T(\epsilon)>0$ such that $0<x'(t)/x(t)<\epsilon$. Integration of this inequality over the interval $[T(\epsilon),t]$ yields
\[
0< \log x(t)-\log x(T(\epsilon))\leq \epsilon(t-T(\epsilon)), \quad t\geq T(\epsilon).
\]
Taking limits as $t\to\infty$, and then as $\epsilon\to 0^+$, gives the second part of \eqref{eq.thmsub1}.
\end{proof}

By strengthening the sublinearity hypothesis on $f$ to \eqref{fasym}, in our first main result we show that solutions of \eqref{functional} grow like those of the autonomous ODE \eqref{ODE}, in the sense that $\lim_{t\to\infty}F(x(t))/Mt = 1$. In the statement of this result, and in the statements of later results, $F$ is the function defined by \eqref{cap_F}.
\begin{theorem}\label{increasing_con}
Suppose that the measures $\mu_1$ and $\mu_2$ obey \eqref{finite_measure}, $f$ obeys \eqref{global_stable}, \eqref{fasym}, and $\psi\in C([-\tau,0];(0,\infty))$. Then, the unique continuous solution, $x$, of \eqref{functional} obeys
\begin{equation} \label{eq.FxtMt1}
\lim_{t \to \infty}x(t) = \infty; \quad
\lim_{t \to \infty}\frac{F(x(t))}{Mt} = 1.
\end{equation}
\end{theorem}
It is notable that, in contrast to linear functional differential equations with a positive measure, the rate of growth is independent 
of the distribution of the mass in the measures $\mu_1$ and $\mu_2$, but depends merely on the overall mass $M=\mu_1([0,\tau])+\mu_2(\mathbb{R}^+)$. Therefore, the growth of solutions cannot be boosted or retarded (at least in terms of the asymptotic relation prescribed in \eqref{eq.FxtMt1}) by greater weight being allocated to more recent values of the solution.  

Now we may simply state specialisations of the above result for the analogous delay and Volterra equations.
\begin{corollary}
Suppose that the measure $\mu_1$ obeys \eqref{finite_measure} with $\mu_2 \equiv 0$, $f$ obeys \eqref{global_stable}, \eqref{fasym}, and $\psi\in C([-\tau,0];(0,\infty))$. Then, the unique continuous solution, $z$, of \eqref{DDE} obeys
\[
\lim_{t \to \infty}z(t) = \infty; \quad
\lim_{t \to \infty}\frac{F(z(t))}{Mt} = 1.
\]
\end{corollary}
\begin{corollary}
Suppose that the measure $\mu_2$ obeys \eqref{finite_measure} with $\mu_1 \equiv 0$, $f$ obeys \eqref{global_stable}, \eqref{fasym}, and $\psi \in (0,\infty)$. Then, the unique continuous solution, $v$, of the Volterra integro-differential equation
\begin{align}\label{VIDE}
v'(t) = \int_{[0,t]}\mu_2(ds)f(v(t-s)),\,\,t \geq 0;\,\, v(0) = \psi,
\end{align}
obeys
\[
\lim_{t \to \infty}v(t) = \infty; \quad
\lim_{t \to \infty}\frac{F(v(t))}{Mt} = 1.
\]
\end{corollary}
At this point it is natural to ask if we can hope to recover asymptotic behaviour similar to that of the solution of \eqref{ODE} if $M =+\infty$. Using the previous result and a comparison argument this can be immediately ruled out. We present this result for the solution of \eqref{functional} but this is perhaps slightly artificial. It is more natural to only consider the unbounded delay component (i.e. $\mu_1 \equiv 0$) and this is essentially how the proof proceeds.
\begin{corollary} \label{Minfinity}
Let $\mu_1$ and $\mu_2$ be non--negative measures on $(\mathbb{R}^+,\mathcal{B}(\mathbb{R}^+))$ satisfying 
\begin{align}\label{infinite_mass}
\int_{[0,\tau]} \mu_1(ds) + \int_{[0,\infty)} \mu_2(ds) = \infty.
\end{align}
Suppose further that $f$ obeys \eqref{global_stable}, \eqref{fasym}, and $\psi \in C([-\tau,0],(0,\infty))$. Then, the unique continuous solution, $x$, of \eqref{functional} obeys
\[
\lim_{t \to \infty}x(t) = \infty; \quad
\lim_{t \to \infty} \frac{F(x(t))}{t} = \infty.
\]
\end{corollary}
The result of Corollary~\ref{Minfinity} can be viewed as a extension of the result of Theorem~\ref{increasing_con} in the limit 
as $M\to\infty$. This can be seen readily by writing \eqref{eq.FxtMt1} in the form 
\[
\lim_{t\to\infty} \frac{F(x(t))}{t}=M,
\]
and now, by letting $M\to\infty$, we get the conclusion of Corollary~\ref{Minfinity}. Of course, Corollary~\ref{Minfinity} tells us 
that the solution of \eqref{functional} now grows more rapidly than that of the ordinary differential equation \eqref{ODE}.

We expect that when the total mass of the measures is infinite, in the sense that \eqref{infinite_mass} holds, this may well give rise to phenomena not captured by relatively crude results such as Corollary \ref{Minfinity}. Treating this issue in more detail will naturally require some additional information about the rate of growth to infinity of the function $M(t):= \int_{[0,t]} \mu_2(ds)$ (set $\mu_1 \equiv 0$). In this spirit, we propose to employ the theory of regular variation to address this delicate, but potentially interesting, case in future work. 

The reader may view the asymptotic relation \eqref{eq.FxtMt1} as giving rather indirect information about the asymptotic behaviour of the solution $x$ of \eqref{functional}, and we might naturally desire more direct information by determining a function $a$ such that $x(t)\sim a(t)$ as $t\to\infty$. Notice in the case of a linear equation that \eqref{eq.FxtMt1} is a statement concerning the Liapunov exponent of a scalar differential equation. Therefore, the direct information we seek constitutes a type of Hartman--Wintner result (cf. Hartman~\cite{H}, and Hartman and Wintner \cite{HW} for ordinary equations with linear leading order terms, and Pituk \cite{pituk} for functional differential equations with linear leading order terms), in contrast to \eqref{eq.FxtMt1}, which is a type of Hartman--Grobman result. A natural candidate for $a$ in this case is 
$a(t)=F^{-1}(Mt)$, and the following Proposition makes this apparent.

\begin{proposition} \label{prop.HWtype}
Let $f$ obey \eqref{fasym} and $F$ be given by \eqref{cap_F}. If $a \in C([0,\infty);(0,\infty))$ is such that $a(t)\sim F^{-1}(Mt)$ as $t\to\infty$, then 
\[
\lim_{t\to\infty} \frac{F(a(t))}{Mt}=1.
\]  
\end{proposition}
As indicated earlier, the proof is given at the end. 

Therefore, Proposition \ref{prop.HWtype} shows that natural direct asymptotic information about the solution gives stronger asymptotic information than the relation \eqref{eq.FxtMt1}. Consequently, it is reasonable to ask if we can impose easily--checked and natural sufficient conditions on  the nonlinear function $f$ so that this can be done. 
The following result gives such conditions under which Theorem \ref{increasing_con} can be appropriately strengthened. 
\begin{theorem}\label{F_inv}
Suppose that the measures $\mu_1$ and $\mu_2$ obey \eqref{finite_measure}, $f$ obeys \eqref{global_stable}, \eqref{fasym}, and $\psi\in C([-\tau,0];(0,\infty))$. Suppose moreover that 
\begin{align}\label{L}
\limsup_{x \to \infty}\frac{f(x)\,F(x)}{x} := L < \infty,
\end{align}
Then, the unique continuous solution, $x$, of \eqref{functional} obeys \eqref{eq.FxtMt1} and \emph{a fortiori}
\begin{equation} \label{eq.xdirectasy}
\lim_{t \to \infty}\frac{x(t)}{F^{-1}(Mt)}=1.
\end{equation}
\end{theorem}
If $f$ is linear we know independently that $z(t)/F^{-1}(Mt)$ does not have zero limit once $\mu_1(\{0\}) + \mu_2(\{0\}) < M$, or in other words, once \eqref{functional} is a true FDE. However, \eqref{L} is merely a sufficient condition to ensure that $z(t) \sim F^{-1}(Mt)$ as $t \to \infty$. This point is elaborated in some detail in the authors recent article \cite{subexp}. 

These caveats notwithstanding \eqref{L} is still a useful condition since it is relatively sharp and does not make overly stringent restrictions on $f$. For example, if 
\begin{multline} \label{eq.fpowerbounded}
\text{There is $\epsilon\in (0,1)$ such that $x\mapsto f(x)/x^{1-\epsilon}$ }\\
\text{is asymptotic to a decreasing function},
\end{multline}
then condition \eqref{L} holds. To see this, let $\varphi$ be the function asymptotic to $x\mapsto f(x)/x^{1-\epsilon}$. Then there is $x_1>1$ such that $x\geq x_1$ implies $\varphi(x)/2 < f(x)/x^{1-\epsilon}<2\varphi(x)$. Then, as $\varphi(u)>\varphi(x)$ for $u<x$, we get for $x\geq x_1$ that
\[
\frac{f(x)}{x}\int_{x_1}^x \frac{1}{f(u)}\,du
\leq \frac{2\varphi(x) x^{1-\epsilon}}{x} \int_{x_1}^x \frac{2}{\varphi(u)u^{1-\epsilon}}\,du
\leq 4 x^{-\epsilon} \int_{x_1}^x \frac{1}{u^{1-\epsilon}}\,du\leq \frac{4}{\epsilon}.
\]  
This gives \eqref{L}, because $x\mapsto f(x)/x$ is bounded on $[1,\infty)$, and therefore so is 
$x\mapsto f(x)/x \cdot \int_1^{x_1} du/f(u)$.

The validity of \eqref{L} within the class of regularly varying functions also casts light on its utility. For example, for any $f \in \text{RV}_\infty(\beta)$ for $\beta \in (0,1)$, \eqref{L} holds: this is a very large class of sublinear functions satisfying \eqref{fasym}. However, if $f \in \text{RV}_\infty(1)$, Karamata's Theorem (cf.~\cite[Theorem 1.5.11]{BGT}) yields 
\[
\lim_{x \to \infty}\frac{f(x)\,F(x)}{x} =\infty,
\]
and so \eqref{L} does not hold in this case. This shows that we cannot relax \eqref{eq.fpowerbounded} to allow $\epsilon=0$.

We make one final remark concerning the condition \eqref{L}. Since $f(x)/x\to 0$ as $x\to\infty$, we have $F(x)\to \infty$ as $x\to\infty$: therefore the possibility arises that $L$ in \eqref{L} could be zero. However, if $f$ obeys 
\eqref{fasym}, we have that $L\geq 1$ and in fact
\begin{equation} \label{Lliminf}
\liminf_{x\to\infty} \frac{F(x)f(x)}{x} \geq 1.
\end{equation}
This is readily seen: by \eqref{fasym}, for every $\epsilon\in (0,1)$, there is $x_1(\epsilon)>0$ 
such that $(1-\epsilon)\phi(x)<f(x)<(1+\epsilon)\phi(x)$ for $x\geq x_1(\epsilon)$, where $\phi\in \mathcal{S}$ 
and so is increasing. Hence for $x_1(\epsilon)\leq u\leq x$ we have 
\[
  f(u)<(1+\epsilon)\phi(u)<(1+\epsilon)\phi(x) < \frac{1+\epsilon}{1-\epsilon}f(x).
\]
Therefore for $x\geq x_1(\epsilon)$
\[
F(x)=F(x_1(\epsilon))+\int_{x_1(\epsilon)}^x \frac{1}{f(u)}\,du \geq F(x_1(\epsilon))+\frac{1-\epsilon}{1+\epsilon}\cdot\frac{x-x_1(\epsilon)}{f(x)}. 
\]
Multiplying by $f(x)/x$, using the fact that this tends to zero as $x\to\infty$, and then taking limits 
as $x\to\infty$, and then as $\epsilon\to 0^+$, we arrive at \eqref{Lliminf}.

\indent Our next result shows that when $f$ is asymptotically decreasing solutions of \eqref{functional} obey $x(t) \sim F^{-1}(Mt)$ as $t \to \infty$ with no additional hypotheses on $f$.
\begin{theorem}\label{decreasing_con}
Suppose that the measures $\mu_1$ and $\mu_2$ obey \eqref{finite_measure}, $f\in C(\mathbb{R}^+;(0,\infty))$ is asymptotic to a decreasing function $\phi \in C^1((0,\infty);(0,\infty))\cap C([0,\infty);(0,\infty))$ and $\psi \in C([-\tau,0];(0,\infty))$. Then, the unique continuous solution, $x$, of 
\eqref{functional} obeys
\[
\lim_{t \to \infty}x(t) = \infty; \quad 
\lim_{t \to \infty}\frac{x(t)}{F^{-1}(Mt)} = 1.
\]
\end{theorem}
We notice that there is no restriction on how rapidly $f$ may decrease in Theorem~\ref{decreasing_con}, in contrast
to the restriction on sublinear increase in $f$ in Theorem~\ref{increasing_con}. Before concluding the section, we give a simple example showing an application of Theorem~\ref{decreasing_con}.

\begin{examples}  
Consider the Volterra equation  
\[
x'(t)=af(x(t))+\int_0^t \frac{1}{(1+t-s)^{\theta+1}} f(x(s))\,ds, \quad t>0; \quad x(0)=\psi>0,
\]
where $a\geq 0$, $\theta>0$ and $f:[0,\infty)\to (0,\infty)$ is locally Lipschitz continuous with 
$f(x)\sim e^{-\alpha x}$ as $x\to\infty$ for $\alpha>0$. The conditions ensure a unique positive continuous solution, and indeed, as $f$ is asymptotic to a decreasing function, we see that all the hypotheses of 
Theorem~\ref{decreasing_con} apply, with 
\[
M=a+\int_0^\infty \frac{1}{(1+u)^{1+\theta}}\,du=a+\frac{1}{\theta}, \quad F(x)\sim \int_1^x e^{\alpha u}\,du=:\Phi(x), \quad \text{as $x\to\infty$}.
\]
It remains to determine explicitly the asymptotic behaviour of $F^{-1}(x)$ as $x\to\infty$. 
Since $\Phi(x)=(e^{\alpha x} -1)/\alpha$, it follows that 
\[
\Phi^{-1}(x)= \frac{1}{\alpha}\log(1+\alpha x).
\]
Therefore $F^{-1}(x)\sim \Phi^{-1}(x)\sim (\log x)/\alpha$ as $x\to\infty$ (see Lemma~\ref{FinvPhiinv}), and so 
by Theorem~\ref{decreasing_con} we get
\begin{equation} \label{eq.xdecexampleasy1}
x(t)\sim F^{-1}(Mt)\sim \frac{1}{\alpha}\log(Mt)\sim \frac{1}{\alpha}\log t
  \quad \text{ as $t\to\infty$}.
\end{equation}

If $f(x)\sim x^{-\beta}$ as $x\to\infty$ for $\beta>0$, we can carry out similar calculations to get 
\[
F^{-1}(x)\sim \left((\beta+1)x\right)^{1/(1+\beta)} \text{ as $x\to\infty$},
\]
so
\begin{equation} \label{eq.xdecexampleasy2}
x(t)\sim F^{-1}(Mt)\sim \left\{(\beta+1) \left(a+\frac{1}{\theta}\right)\right\}^{1/(1+\beta)} t^{1/(1+\beta)}
  \quad \text{ as $t\to\infty$}.
\end{equation}
\end{examples}

We close this section by noting that we can generalise the results presented here to the case of finitely many nonlinear functions, in the spirit of Pinto \cite{pinto1990integral}. Consider the following functional equation with finitely many distinct nonlinearities and measures present
\begin{align}\label{finitely_many}
x'(t) &= \sum_{j=1}^{N_1}\int_{[0,\tau]} \bar{\mu}_j(ds)\bar{f}_j(x(t-s)) + \sum_{j=1}^{N_2}\int_{[0,t]} \underline{\mu}_j(ds)\underline{f}_j(x(t-s)), \,\, t>0,\\
x(t) &= \psi(t), \,\, t \in [-\tau,0],\,\, \tau>0, \nonumber
\end{align}
where we interpret the relevant sum as zero in the case that either $N_1$ or $N_2$ are zero. We also assume that $N_1,\,N_2 \in \mathbb{N}$ with $\max(N_1,\,N_2)>0$ to avoid trivialities. There is no loss of generality in assuming that the finite measures $\bar{\mu}_j$ are defined on the same compact interval $[0,\tau]$, for if the support of the measures were different, we could simply take $\tau$ to be the largest among the support lengths, and extend measures to the common support by setting them to be zero off their natural support. To deal with an equation of the form \eqref{finitely_many} using the framework and methods of this paper we will naturally require analogues of conditions \eqref{finite_measure} and \eqref{fasym}. Hence, with $\mathcal{S}$ defined by \eqref{def.smoothsublinear}, we will suppose that the nonlinear functions $\bar{f}_1,\dots,\bar{f}_{N_1},\underline{f}_1,\dots,\underline{f}_{N_2}$ obey
\begin{align}
&\text{There exists }\phi \in \mathcal{S}\text{ such that }\bar{f}_j(x) \sim \bar{\lambda}_j\,\phi(x) \text{, for } j \in 1,\dots,N_1 \text{ and}\nonumber\\
&\underline{f}_j(x) \sim \underline{\lambda}_j\,\phi(x) \text{, for } j \in 1,\dots,N_2, \text{ with } \bar{\lambda}_1, \dots, \bar{\lambda}_{N_1}, \underline{\lambda}_1, \dots, \underline{\lambda}_{N_2} \in \mathbb{R}^+,
\label{fasym_mult}
\end{align}
and that the measures $\bar{\mu}_1,\dots,\bar{\mu}_{N_1},\underline{\mu}_1,\dots,\underline{\mu}_{N_2}$ obey
\begin{align}\label{finite_mult_measures}
\sum_{j=1}^{N_1}\int_{[0,\tau]} \bar{\mu}_j(ds)\mathbbm{1}_{\{\bar{\lambda}_j>0\}} + \sum_{j=1}^{N_2}\int_{[0,\infty)} \underline{\mu}_j(ds)\mathbbm{1}_{\{\underline{\lambda}_j>0\}} =: M \in (0,\infty),
\end{align}
with $\bar{\lambda}_1,\dots,\bar{\lambda}_{N_1},\underline{\lambda}_1,\dots,\underline{\lambda}_{N_2}$ defined by \eqref{fasym_mult}. The positivity of $M$ ensures that there is at least one nonlinearity with leading order behaviour $\phi$, and that all other nonlinearities have no faster rate of growth.  

In essence, the following theorem says that one only need consider the fastest growing nonlinearity if several nonlinear functions and measures are present, and apply our earlier results to this reduced equation. Without much work the above results can be used to prove the following theorem: accordingly, we state it without proof. 
\begin{theorem}\label{Thm.incr.mult}
Let $\bar{\mu}_1,\dots,\bar{\mu}_{N_1},\underline{\mu}_1,\dots,\underline{\mu}_{N_2}$ be non-negative Borel measures satisfying \eqref{finite_mult_measures}. Suppose further that\\ $\bar{f}_1,\dots,\bar{f}_{N_1},\underline{f}_1,\dots,\underline{f}_{N_2} \in C(\mathbb{R}^+;(0,\infty))$ satisfy \eqref{fasym_mult} with 
\[
\min(\bar{\lambda}_1,\dots,\bar{\lambda}_{N_1},\underline{\lambda}_1,\dots,\underline{\lambda}_{N_2})>0
\] 
and $\psi \in C([-\tau,0];(0,\infty))$. Then, the unique continuous solution, $x$, of \eqref{finitely_many}
obeys
\[
\lim_{t \to \infty}x(t) = \infty, \quad \lim_{t \to \infty}\frac{F(x(t))}{Mt} = 1,
\]
where $\lambda:= \bar{\lambda}_1 + \dots + \bar{\lambda}_{N_1}+ \underline{\lambda}_1 + \dots + \underline{\lambda}_{N_2}$, $f(x) := \lambda\,\phi(x)$, $M$ is defined by \eqref{finite_mult_measures} and $F$ is defined by \eqref{cap_F}.
\end{theorem}
We can formulate an analogous result to Theorem \ref{Thm.incr.mult} in the case when we have finitely many decreasing nonlinear functions in \eqref{finitely_many} using Theorem \ref{decreasing_con}, \eqref{finite_mult_measures} and a condition of the form \eqref{fasym_mult}.
\section{Results with Regular Variation} \label{sec.RV}
We now present some auxiliary results which show that our main results can readily be applied to the case when the sublinear function $f$ is regularly varying at infinity, in Karamata's sense. 
The standard text on the theory of regularly varying functions is that of Bingham, Goldie and Teugels \cite{BGT} and all properties and results which we exploit can be found therein. A more succinct outline of the main useful properties (within the present context) can be found in  \cite{appleby2013classification}.

Essentially, if $f$ is in $\text{RV}_\infty(\beta)$ with $\beta\in (0,1)$, it immediately satisfies the condition \eqref{fasym}, and so Theorem~\ref{increasing_con} and all relevant corollaries can be applied. If $\beta<0$, then the hypothesis that $f$ is asymptotically decreasing in Theorem~\ref{decreasing_con}
is satisfied, and so Theorem~\ref{decreasing_con} applies. If $\beta>1$, $f$ is not sublinear, and we are outside the scope of the paper. The case when $\beta\in \{0,1\}$ contains subtleties which we discuss presently, but in some cases we may still apply our results easily. 

Due to Karamtata's theorem, and the theory of asymptotic inverses for regularly varying functions, in many cases determining 
the asymptotic behaviour of $F$ or $F^{-1}$ is straightforward, and so explicit rates of growth for the solution can be established. We give an example below in which we do these calculations directly, rather than extracting them from known results, in order to keep our presentation self--contained.  

Finally, we note that there is a burgeoning literature regarding the application of the theory of regular variation to the 
asymptotic behaviour of ordinary and functional differential equations (see for example the monographs of Mari\'{c} \cite{maric2000regular}, and {\v{R}}eh{\'a}k \cite{rehakrv} and recent representative papers such as those of Chatzarakis et al, \cite{chatzarakis2014precise}, Matucci and {\v{R}}eh{\'a}k \cite{matucci2014asymptotics,matucci2015extremal}, and Takasi and Manojlovi{\'c} \cite{takasi2011precise}).

Our first result is a direct application of Theorems \ref{increasing_con} and \ref{decreasing_con}, in conjunction with 
Theorem~\ref{F_inv}.
\begin{theorem}\label{RV_beta_corollary_1}
Suppose that the measures $\mu_1$ and $\mu_2$ obey \eqref{finite_measure}, $f \in \text{RV}_\infty(\beta), \,\, \beta \in (-\infty,1)/\{0\}$ and $\psi \in C([-\tau,0];(0,\infty))$. Then, the unique continuous solution, $x$, of \eqref{functional} obeys
\[ 
\lim_{t \to \infty}x(t) = \infty; \quad \lim_{t \to \infty}\frac{x(t)}{F^{-1}(Mt)} = 1.
\]
\end{theorem}
\begin{proof}
	If $\beta \in (-\infty,0)$ we immediately have that $f$ is asymptotic to a decreasing function and hence we may apply Theorem \ref{decreasing_con} to yield the claim. If $\beta\in (0,1)$ there exists an increasing function $\phi \in C^1((0,\infty);(0,\infty)) \cap \text{RV}_\infty(\beta)$ such that 
	\[
	\lim_{x \to \infty}\frac{f(x)}{\phi(x)} = 1, \quad \lim_{x \to \infty}\frac{x\,\phi'(x)}{\phi(x)}=\beta.
	\]
	It follows that $\phi'(x) \sim \beta\,\phi(x)/x$ as $x\to\infty$ and hence that $\phi' \in \text{RV}(\beta-1)$. Therefore $\lim_{x\to\infty}\phi'(x)= 0$ \cite[Proposition 1.5.1]{BGT}. Now apply Theorem \ref{increasing_con} to obtain $\lim_{t\to\infty}F(x(t))/Mt=1$; we use Theorem \ref{F_inv} to strengthen this conclusion. By Karamata's Theorem \cite[Theorem 1.5.11]{BGT} we obtain 
	\[
	\limsup_{x \to \infty}\frac{f(x)\,F(x)}{x} = 1-\beta < \infty.
	\]
	Therefore applying Theorem \ref{F_inv} yields $x(t) \sim F^{-1}(Mt)$ as $t \to \infty$.
\end{proof}
\begin{examples}
The following is but a simple application of Theorem~\ref{RV_beta_corollary_1}, and the reader is invited to consider others.   
Consider the Volterra equation  
\[
x'(t)=af(x(t))+\int_0^t \frac{1}{(1+t-s)^{\theta+1}} f(x(s))\,ds, \quad t>0; \quad x(0)=\psi>0,
\]
where $a\geq 0$, $\theta>0$ and $f:[0,\infty)\to (0,\infty)$ is locally Lipschitz continuous with 
$f(x)\sim x^\beta (\log x)^\alpha$ as $x\to\infty$ for $\beta\in (0,1)$, $\alpha\in \mathbb{R}$. The conditions ensure a unique positive continuous solution, and indeed, as $f\in \text{RV}_\infty(\beta)$, we see that all the hypotheses of Theorem~\ref{RV_beta_corollary_1} apply, with 
\[
M=a+\int_0^\infty \frac{du}{(1+u)^{1+\theta}}=a+\frac{1}{\theta}, \quad F(x)\sim \int_e^x \frac{du}{u^\beta (\log u)^\alpha}=:\Phi(x), \quad \text{as $x\to\infty$}.
\]
It remains to determine explicitly the asymptotic behaviour of $F^{-1}(x)$ as $x\to\infty$. 
Clearly
\[
\Phi(x)=\int_1^{\log x} v^{-\alpha} e^{(1-\beta)v}\,dv.
\]
By applying L'H\^opital's rule, we get
\[
\int_1^y v^{-\alpha} e^{(1-\beta)v}\,dv \sim \frac{1}{1-\beta} y^{-\alpha} e^{(1-\beta)y}, \quad \text{as $y\to\infty$},
\]
and so we have 
\[
F(x)\sim \frac{1}{1-\beta} (\log x)^{-\alpha} x^{1-\beta}, \quad \text{as $x\to\infty$}.
\]
Using this asymptotic relation (with $x=F^{-1}(y)$), it can readily be shown that 
$\log F^{-1}(y)/\log y\to 1/(1-\beta)$ as $y\to\infty$. Replacing this in the asymptotic relation for $F(x)$ leads to  
\[
F^{-1}(y) \sim (1-\beta)^{\frac{1-\alpha}{1-\beta}} (\log y)^\frac{\alpha}{1-\beta} y^{\frac{1}{1-\beta}}, \quad \text{as $y\to\infty$}.
\]
Now by Theorem~\ref{RV_beta_corollary_1} we get
\begin{equation} \label{eq.xRVexampleasy}
x(t)\sim F^{-1}(Mt)\sim 
 (1-\beta)^{\frac{1-\alpha}{1-\beta}}\left(a+\frac{1}{\theta}\right)^{\frac{1}{1-\beta}}
(\log t)^\frac{\alpha}{1-\beta}   t^{\frac{1}{1-\beta}}, \quad \text{ as $t\to\infty$}.
\end{equation}
\end{examples}

\begin{examples}
In the last example, the measure exhibited power--law decay. We consider now the same nonlinearity, but an exponentially decaying measure, so that the Volterra equation reads,
\[
x'(t)=af(x(t))+\int_0^t e^{-\theta(t-s)} f(x(s))\,ds, \quad t>0; \quad x(0)=\psi>0,
\]
where once again $a\geq 0$ and $\theta>0$. 
As before, there is a unique positive continuous solution and all the hypotheses of 
Theorem~\ref{RV_beta_corollary_1} apply, with 
\[
M=a+\int_0^\infty e^{-\theta u}\,du=a+\frac{1}{\theta},
\]
so we recover exactly the same asymptotic behaviour of the solution $x$ as in the last example 
(i.e., \eqref{eq.xRVexampleasy}). Therefore, even though the past behaviour of the solution is discounted much more rapidly in this example than in the previous one, there is no difference in the rate of growth of the solution (to first order) because the value of $M$ is the same in each case. Indeed, if we were to consider the delay--differential equation
\[
x'(t)=af(x(t))+\frac{1}{\theta}f(x(t-\tau)), \quad t>0; \quad x(t)=\psi(t)>0, \quad t\in [-\tau,0],
\]
with the same $f$, and $\tau>0$ fixed, we see once again $x$ obeys \eqref{eq.xRVexampleasy}. This is 
because the mass of the point delta measure at $\tau$ is $1/\theta$, and the mass of the point delta measure at $0$ is 
$a$, so $M=a+1/\theta$, just as before. In this case, even though the past behaviour of the solution makes no contribution before time $t-\tau$, the same growth rate eventuates.
\end{examples}

\begin{examples}
Notice that Theorem~\ref{RV_beta_corollary_1} does not apply to the case where $\beta=1$. However, if $f$ is truly sublinear, 
then the condition \eqref{fasym} holds, and Theorem~\ref{increasing_con} applies. However, as mentioned earlier, $f$ cannot satisfy \eqref{L}, and so we cannot conclude directly that $x$ obeys \eqref{eq.xdirectasy}. Indeed, it has been shown in \cite[Theorem 2.2]{subexp}, in the case that $f'\in \text{RV}_{\infty}(0)$ (which implies $f\in \text{RV}_\infty(1)$), and the delay is bounded (so $\mu_2\equiv 0$), that we have 
\[
\lim_{t\to\infty} \frac{x(t)}{F^{-1}(Mt)}=e^{-\lambda\,C},
\]  
 where 
\[
\lim_{x\to\infty} \frac{f(x)}{x/\log x}=:\lambda\in [0,\infty], \quad C:= \int_{[0,\tau]} s\,\mu_1(ds).
\]

Therefore, the conclusion of Theorem~\ref{RV_beta_corollary_1} \emph{need not} hold if $f(x)$ is of larger order than $x/\log x$ as $x\to\infty$, and the delay is nontrivial, although $x(t)/F^{-1}(Mt)\to 1$ as $t\to\infty$ if $f(x)=o(x/\log x)$ as $x\to\infty$. It would obviously be of interest to see whether the rather restrictive slow variation hypothesis on $f'$ in \cite{subexp} could be replaced by weaker monotonicity--type conditions, and to what degree the result still holds for Volterra equations.  

The determination of the asymptotic behaviour of $F$ and $F^{-1}$ is more delicate when $f\in \text{RV}_\infty(1)$, in large part because Karamata's theorem only shows that $1/F(x)=o(f(x)/x)$ as $x\to\infty$. However, we supply a concrete example in which 
the asymptotic behaviour of $F$ can be worked out explicitly, and Theorem~\ref{increasing_con} applies. 

We consider the Volterra equation  
\[
x'(t)=\int_0^t \frac{1}{(1+t-s)^{\theta+1}} f(x(s))\,ds, \quad t>0; \quad x(0)=\psi>0,
\]
where $\theta>0$ and $f:[0,\infty)\to (0,\infty)$ is locally Lipschitz continuous with $f(x)\sim x/(\log x)^\alpha$ as $x\to\infty$ for $\alpha>0$. We see that $f\in \text{RV}_\infty(1)$ and $f(x)/x\to 0$ as $x\to\infty$, so not only do these conditions ensure a unique positive continuous and growing solution, but moreover they ensure that Theorem~\ref{increasing_con} apply, with 
\[
M=\int_0^\infty \frac{1}{(1+u)^{1+\theta}}\,du=\frac{1}{\theta}, \quad F(x)\sim \int_e^x \frac{(\log u)^\alpha}{u}\,du=:\Phi(x), \quad \text{as $x\to\infty$}.
\]
Clearly
\[
\Phi(x)=\frac{(\log x)^{\alpha+1}}{\alpha+1},
\]
and so we have 
\[
F(x)\sim \frac{(\log x)^{\alpha+1}}{\alpha+1}, \quad \text{as $x\to\infty$}, 
\]
Therefore by Theorem~\ref{increasing_con},
\[
\lim_{t\to\infty} \frac{1}{t} \frac{(\log x(t))^{\alpha+1}}{\alpha+1}=\frac{1}{\theta},
\]
so
\[
\lim_{t\to\infty} \frac{\log x(t)}{t^{1/(\alpha+1)}} =\left(\frac{\alpha+1}{\theta}\right)^{1/(\alpha+1)}.
\]
Since $\alpha>0$, the growth is slower than exponential, as expected, and we may view this as the limit as a generalisation of the Liapunov exponent in this nonlinear setting. 
\end{examples} 

\begin{examples}
To illustrate the utility of only requiring asymptotic monotonicity in our earlier results suppose $f(x) = x^\alpha [2+\sin(\log_2(x+2))]$, $\alpha \in (0,1)$. $f\in \text{RV}_\infty(\alpha)$ and, although $f$ is clearly non-monotone, it oscillates slowly enough that $f'(x)>0$ for all $x$ large enough.
\end{examples}
Theorem \ref{RV_beta_corollary_1} immediately raises the question of what happens when $f$ is regularly varying with index zero. In this case there is no guarantee that $f$ will be asymptotic to a monotone function and hence we cannot rely on any of our previous work. An example emphasising the extreme oscillatory behaviour possible within the class $\text{RV}_\infty(0)$ is to take 
\begin{equation}\label{RV_osc}
f(x) = \exp[ \ln(2+x)^{\tfrac{1}{3}} \cos(\ln(2+x)^{\tfrac{1}{3}}) ].
\end{equation}
In this example, $\liminf_{x \to \infty}f(x) = 0$ and $\limsup_{x \to \infty}f(x) = \infty$. \\
\indent The following pair of results provide a partial answer to what happens when we have $f \in \text{RV}_\infty(0)$. Our first result shows that when the delay is bounded we still have $x(t) \sim F^{-1}(Mt)$ as $t \to \infty$ with no additional hypotheses on the problem.
\begin{theorem} \label{RV_0}
Suppose that the measure $\mu_1$ obeys \eqref{finite_measure} with $\mu_2 \equiv 0$, $f \in \text{RV}_\infty(0)$ and $\psi \in C([-\tau,0];(0,\infty))$. Then, the unique continuous solution, $z$, of \eqref{DDE} obeys
\[
\lim_{t \to \infty}z(t) = \infty; \quad \lim_{t \to \infty}\frac{z(t)}{F^{-1}(Mt)} = 1.
\]
\end{theorem}
In the case of unbounded delay the problem is much more delicate and only under additional hypotheses have we been able to retain the asymptotic rates as before. If we assume that $f$ is bounded away from zero we rule out highly irregular nonlinearities such as \eqref{RV_osc} and we can prove the following result. 
\begin{theorem}\label{RV_0_2}
Suppose that the measure $\mu_2$ obeys \eqref{finite_measure} with $\mu_1 \equiv 0$, $f \in \text{RV}_\infty(0)$ is bounded away from zero and $\psi > 0$. Then, the unique continuous solution, $v$, of \eqref{VIDE} obeys
\[
\lim_{t \to \infty}v(t) = \infty; \quad \lim_{t \to \infty}\frac{v(t)}{F^{-1}(Mt)} = 1.
\]
\end{theorem}
Of course, the hypothesis in Theorem~\ref{RV_0_2} that $f$ is bounded away from zero (by continuity of $x\mapsto f(x)$, this lower bound is meaningful in the limit as $x\to\infty$) is satisfied in the case that $f$ is asymptotically monotone. Therefore, one can 
rephrase Theorem~\ref{RV_beta_corollary_1} to include the case that $\beta=0$, at the small expense of assuming the asymptotic monotonicity of $f$ (which is automatically true when $\beta>0$). 

The above results constitute our main attempt to deal with the case when the nonlinear function has no monotonicity property whatsoever and as such we think it instructive to see how far this calculation can be taken in the case of unbounded delay, without additional hypotheses. In fact, the following lemma shows that we can obtain a sharp lower bound.
\begin{theorem}\label{RV_0_lower_bound}
Suppose that the measure $\mu_2$ obeys \eqref{finite_measure} with $\mu_1 \equiv 0$, $f \in \text{RV}_\infty(0)$ and $\psi > 0$. Then, the unique continuous solution, $v$, of \eqref{VIDE} obeys
\[
\lim_{t \to \infty}v(t) = \infty; \quad \liminf_{t \to \infty}\frac{v(t)}{F^{-1}(Mt)} \geq 1.
\]
\end{theorem}
Our final result demonstrates that under no additional assumptions we can at least obtain a `crude' upper bound on the growth rate of the solution to \eqref{VIDE} which agrees with the lower bound provided by Theorem \ref{RV_0_lower_bound} up to a logarithmic factor.
\begin{theorem}\label{RV_0_upper}
Suppose that the measure $\mu_2$ obeys \eqref{finite_measure} with $\mu_1 \equiv 0$, $f \in \text{RV}_\infty(0)$ is bounded away from zero and $\psi > 0$. Then, the unique continuous solution, $v$, of \eqref{VIDE} obeys
\[
\lim_{t \to \infty}v(t) = \infty; \quad \lim_{t \to \infty}\frac{\log(v(t))}{\log(t)} = 1.
\]
\end{theorem}

\section{Examples of Sublinearity}\label{sec.examples}
Before giving proofs of our results in Section \ref{sec.mainresults}, we close with the examples promised in Section \ref{discussion} which show the scope of the strengthened sublinearity hypothesis, \eqref{fasym}, frequently imposed on $f$. 

We find for the purposes of these examples it is more natural and instructive to construct an $f$ with the desired properties by  specifying $f'$. We defer the justification of the following examples to Section \ref{sec_exp}. Throughout these examples we define $f'$, for $n \in \mathbb{N}$, as follows
\begin{equation}\label{f_prime}
f'(x) = 
\begin{cases}
\eta(x), \,\, x \in (0,1] \cup (n+w_n,n+1],\\
\eta(n) + \frac{2(x-n)(h_n-\eta(n))}{w_n}, \,\, x \in (n,n+w_n/2],\\
h_n + \frac{2(x-n-w_n/2)(\eta(n+w_n)-h_n)}{w_n}, \,\, x \in (n+w_n/2,n+w_n]. 
\end{cases}
\end{equation}
Choosing $\eta(x) > 0$ for all $x>0$ and $h_n > 0$ for all $n\in \mathbb{N}$ ensures that $f$ is strictly increasing. Define $\phi(x) := \int_0^x \eta(u)du$ and by construction we will have $\phi \sim f$. In order to have both $\liminf_{x\to\infty}f'(x)=0$ and $\limsup_{x\to\infty}f'(x)>0$ we want $f'$ to largely follow the behaviour of $\eta$, which tends to zero, but to also have high, narrow spikes inherited from $h_n$. 
\begin{examples}\label{first_example}
Suppose $f'$ is defined by \eqref{f_prime}, $\phi(x):= \int_0^x \eta(u)du$ and that $\eta(x) \downarrow 0$ as $x\to\infty$, $0 < w_n < 1$, and $h_n > \phi'(n)$ for all $n \in \mathbb{N}$. Furthermore suppose that
\[
\lim_{x\to\infty}\phi(x)=\infty, \, \lim_{n\to\infty}h_n= L \in (0,\infty], \, \lim_{n\to\infty}\sum_{j=1}^n\frac{w_j h_j}{\phi(n)}=0, \, \lim_{n\to\infty}\sum_{j=1}^n\frac{w_j \phi'(j)}{\phi(n)}=0.
\] 
Then
\begin{enumerate}[(i.)]
\item $\liminf_{x\to\infty}f'(x)=0,\,\limsup_{x\to\infty}f'(x)\geq L$.
\item $f(x) \sim \phi(x)$ as $x\to\infty$ and hence $\lim_{x\to\infty}f(x)/x=0$.
\end{enumerate}
\end{examples}
The $f$ constructed in Example \ref{first_example} has ``spikes'' in its derivative which can grow arbitrarily quickly but since it is asymptotic to $\phi$ it still obeys condition \eqref{fasym}. When $\phi$ tends to a finite limit so does $f$ and moreover we do not require that $\phi$ grows faster than the sums of $w_j\,h_j$ and $h_j\,\phi'(j)$.
\begin{examples}\label{second_example}
Suppose $f'$ is defined by \eqref{f_prime}, $\phi(x):= \int_0^x \eta(u)du$ and that $\eta(x) \downarrow 0$ as $x\to\infty$, $0 < w_n < 1$, and $h_n > \phi'(n)$ for all $n \in \mathbb{N}$. Furthermore, if $L^*, \,L_0$ and $L_1$ are finite, suppose that
\[
\lim_{x\to\infty}\phi(x)=L^*, \, \lim_{n\to\infty}h_n= L \in (0,\infty], \, \lim_{n\to\infty}\sum_{j=1}^n w_j h_j=L_0, \, \lim_{n\to\infty}\sum_{j=1}^n w_j \phi'(j)=L_1.
\] 
Then
\begin{enumerate}[(i.)]
\item $\liminf_{x\to\infty}f'(x)=0,\,\limsup_{x\to\infty}f'(x)\geq L$.
\item $f(x)\to L' \in (0,\infty)$ as $x\to\infty$ and hence $\lim_{x\to\infty}f(x)/x=0$.
\end{enumerate}
\end{examples}
In this case $f$ is asymptotic to a constant so it once more obeys \eqref{fasym}.
\section{Proofs with Increasing Nonlinearity} \label{sec.mainresults}
In this section we prove general results in which it is assumed that $f$ is asymptotic to an increasing function, 
often in the class $\mathcal{S}$ introduced in \eqref{def.smoothsublinear}.
 
Before giving the proofs of our main results we state and prove some useful technical lemmata. The first makes explicit the fact that \eqref{fasym} implies sublinearity of $f$.
\begin{lemma}\label{to_zero}
Suppose $f$ is a continuous function obeying \eqref{fasym}.
Then $f(x)/x \to 0$ as $x \to \infty$.
\end{lemma}
\begin{proof}
Since $\phi$ is increasing, either $\phi(x) \to \infty$ as $x \to \infty$ or $\phi(x) \to L \in (0,\infty)$ as $x \to \infty$. In the latter case the asymptotic equivalence of $\phi$ and $f$ yields $\lim_{x \to \infty}f(x)/x=0$. In the first case we use L'Hopitals rule to obtain
\[
\lim_{x \to \infty}\phi(x)/x = \lim_{x \to \infty}\phi'(x)= 0.
\]
Thus
$
\lim_{x \to \infty}f(x)/x = \lim_{x \to \infty}\left( f(x)/\phi(x)\right)\left(\phi(x)/x\right) = 0.
$
\end{proof}
The proof of Proposition~\ref{prop.HWtype} requires the following preliminary lemma.
\begin{lemma}\label{asym_equiv_copy}
Suppose $\varphi$ is such that $\varphi(x)\to\infty$ as $x\to\infty$, $\varphi'(x)>0$ for $x>0$ and $\varphi'(x)$ is decreasing with $\varphi'(x)\to 0$ as $x\to\infty$. If $b,c \in C(\mathbb{R^+},\mathbb{R^+})$ obey 
$\lim_{t\to\infty}b(t)=\lim_{t\to\infty}c(t) = \infty$, and $b(t) \sim c(t)$ as $t\to\infty$, then $\varphi(b(t)) \sim \varphi(c(t))$ as $t\to\infty$.
\end{lemma}
\begin{proof}[Proof of Lemma \ref{asym_equiv_copy}]
We start by showing, for every $\Lambda>1$ that
\begin{equation} \label{eq.firstphi}
\limsup_{x\to\infty} \frac{\varphi(\Lambda x)}{\varphi(x)}\leq \Lambda.
\end{equation}
We suppose throughout that $x \geq a >0$. Then $\varphi(x) - \varphi(a) = \int_a^x \varphi'(u)du \geq \varphi'(x)(x-a)$. Thus 
\begin{align}\label{phi_prime}
\limsup_{x\to\infty}\frac{\varphi'(x)x}{\varphi(x)} = \limsup_{x\to\infty}\frac{\varphi'(x)(x-a)}{\varphi(x)}\frac{x}{x-a} 
&\leq \limsup_{x\to\infty}\frac{\varphi(x)-\varphi(a)}{\varphi(x)} = 1.
\end{align}
 To prove \eqref{eq.firstphi} we proceed as follows:
\begin{align*}
\frac{\varphi(\Lambda x)}{\varphi(x)} &= \frac{\int_a^{\Lambda x} \varphi'(u)du + \varphi(a)}{\varphi(x)} = \frac{\int_a^{x} \varphi'(u)du + \int_x^{\Lambda x} \varphi'(u)du + \varphi(a)}{\varphi(x)} \\ 
&= 1 + \frac{\int_x^{\Lambda x} \varphi'(u)du}{\varphi(x)} \leq 1 + (\Lambda-1)\frac{\varphi'(x)\,x}{\varphi(x)}.
\end{align*}
Now taking the limsup, and using \eqref{phi_prime}, we have shown \eqref{eq.firstphi}.

We are now in a position to prove our claim. By hypothesis, for all $\epsilon>0$ there exists $T(\epsilon)>0$ such that for all $t \geq T(\epsilon)$
\[
(1-\epsilon)c(t) < b(t) < (1+\epsilon)c(t).
\]
Monotonicity of $\varphi$ immediately yields
\[
\frac{\varphi((1-\epsilon)c(t))}{\varphi(c(t))} < \frac{\varphi(b(t))}{\varphi(c(t))} 
< \frac{\varphi((1+\epsilon)c(t))}{\varphi(c(t))}, \quad t \geq T.
\]
By \eqref{eq.firstphi}, and the divergence of $c$, there exists $T' > T$ such that $\varphi((1+\epsilon)c(t)) < (1+\epsilon)^2 \varphi(c(t))$ for all $t \geq T'$. Hence $\limsup_{t\to\infty}\varphi(b(t))/\varphi(c(t)) \leq 1.$ Reversing the roles of $b$ and $c$ in the above argument we have that 
\[
\limsup_{t\to\infty}\varphi(c(t))/\varphi(b(t)) \leq 1,
\] 
or equivalently, $\liminf_{t\to\infty}\varphi(b(t))/\varphi(c(t)) \geq 1$, completing the proof.
\end{proof}

\begin{proof}[Proof of Proposition~\ref{prop.HWtype}]
By \eqref{fasym}, we have that $\Phi(x)=\int_1^x du/\phi(u)$ obeys $\Phi(x)\sim F(x)$ as $x\to\infty$. Notice also from 
\eqref{fasym} that $\Phi$ is increasing with decreasing derivative. Now, we apply Lemma~\ref{asym_equiv_copy} with $\varphi=\Phi$, so that if $b$ and $c$ are continuous functions with $b(t)\to\infty$ and $b(t)\sim c(t)$ as $t\to\infty$, then 
\[
\Phi(b(t))\sim \Phi(c(t)) \text{ as $t\to\infty$}.
\]
Therefore, it follows that $\Phi(b(t))\sim F(c(t))$ as $t\to\infty$. Now take $c(t)=F^{-1}(Mt)$ and $b(t)=a(t)$, so that 
$\Phi(a(t))/Mt\to 1$ as $t\to\infty$. Since $\Phi(x)\sim F(x)$ as $x\to\infty$ the claim follows.
\end{proof}

\begin{proof}[Proof of Theorem \ref{increasing_con}]
From our positivity hypotheses on $\mu_1, \mu_2$  and $\psi$ we have
\[
x'(0) = \int_{[0,\tau]} \mu_1(ds)f(x(-s)) = \int_{[0,\tau]} \mu_1(ds)f(\psi(-s))\geq 0.
\]
Suppose there exists a number $t_0 > 0$ such that $x'(t) \geq 0$ for all $t \in [0,t_0)$ and $x'(t_0)<0$. Hence $x(t) \geq x(0) =\psi(0) > 0$ for all $t \in [0,t_0]$ and we have 
\[
x'(t_0) = \int_{[0,\tau]} \mu_1(ds)f(x(t_0-s)) + \int_{[0,t_0]} \mu_2(ds) f(x(t_0-s)) > 0.
\]
Thus $x'(t) \geq 0$ for all $t \geq 0$ and $x(t) > 0$ for all $t \geq -\tau$. Thus, because at least one of $\mu_1$ and $\mu_2$ is non-trivial, there exists $T>0$ such that $x'(t)>0$ for all $t \geq T$. Therefore either $x(t) \to \infty$ as $t \to \infty$ or $x(t) \to L \in (0,\infty)$ as $t \to \infty$. If the second possibility prevails, then taking limits across \eqref{functional} gives
\[
\lim_{t \to \infty}x'(t) = M f(L) > 0.
\]
But this means that $\lim_{t \to \infty}x(t) = \infty$, a contradiction.

\emph{Step 1:} We first compute the required upper bound on the growth rate of the solution. If $\epsilon>0$ is arbitrary, by hypothesis, there exists $x_1(\epsilon)$ such that for all $x > x_1(\epsilon)$, 
$
(1-\epsilon)\phi(x) < f(x) < (1+\epsilon)\phi(x).
$
Hence, since $\lim_{t\to\infty}x(t)=\infty$, there exists $T_1(\epsilon)$ such that for $t \geq T_1(\epsilon)$, $x(t) > x_1(\epsilon)$. Thus for all $t \geq T(\epsilon) := T_1(\epsilon)+ \tau$ we have
\begin{align*}
x'(t) &= \int_{[0,\tau]} \mu_1(ds)f(x(t-s)) + \int_{(t-T,t]} \mu_2(ds) f(x(t-s)) \\
&\qquad+ \int_{[0,t-T]} \mu_2(ds) f(x(t-s)) \\
&\leq (1+\epsilon)\int_{[0,\tau]} \mu_1(ds)\phi(x(t-s)) + (1+\epsilon) \int_{[0,t-T]} \mu_2(ds) \phi(x(t-s)) + R(t),
\end{align*}
where $R(t) := \int_{(t-T,t]}\mu_2(ds)f(x(t-s))$. 
Now the monotonicity of the solution and $\phi$ yield
\begin{align*}
x'(t) &\leq (1+\epsilon) \left(\int_{[0,\tau]} \mu_1(ds)\phi(x(t)) 
+ \int_{[0,\infty)} \mu_2(ds)\phi(x(t)) \right) + R(t)\\ 
&= (1+ \epsilon)\,M\, \phi(x(t)) + R(t).
\end{align*}
Hence 
\begin{align}\label{lower}
\frac{x'(t)}{\phi(x(t))} \leq (1+ \epsilon)M + \frac{R(t)}{\phi(x(t))}.
\end{align}
We estimate the final term on the right-hand side by
\[
\frac{R(t)}{\phi(x(t))} = \frac{\int_{(t-T,t]}\mu_2(ds)f(x(t-s))}{\phi(x(t))} 
\leq \frac{\int_{(t-T,t]}\mu_2(ds)}{\phi(x(t))} \sup_{u \in [0,T]}f(x(u)).
\]
Since $f \circ x$ is a continuous function the supremum is bounded on compact intervals and we have that $\lim_{t \to \infty} R(t)=0$. Also, $\phi$ nondecreasing and $x(t) \to \infty$ as $t \to \infty$ mean that $\lim_{t \to \infty}\phi(x(t)) \in (0,\infty]$ and hence $\lim_{t \to \infty}R(t)/ \phi(x(t))=0$. Returning to \eqref{lower} we may take the limsup to obtain
$
\limsup_{t \to \infty}x'(t)/\phi(x(t)) \leq (1+ \epsilon)M.
$
Asymptotic integration then yields
$
\limsup_{t \to \infty}\Phi(x(t))/Mt \leq 1,
$
and the asymptotic equivalence of $F$ and $\Phi$ allow us to conclude that
$
\limsup_{t \to \infty}F(x(t))/Mt \leq 1.
$

\emph{Step 2:} We now compute the corresponding lower bound. Define
\[
\mu_1 := \int_{[0,\tau]}\mu_1(ds),\quad \mu_2 := \int_{[0,\infty)}\mu_2(ds).
\] 
By \eqref{finite_measure}, for an arbitrary $\epsilon \in (0,1)$, there exists $T_2(\epsilon)$ large enough that
\[
(1-\epsilon) \int_{[0,\infty)} \mu_2(ds) \leq \int_{[0,T_2]}\mu_2(ds) \leq \int_{[0,\infty)} \mu_2(ds).
\]
Furthermore, since $\lim_{x \to \infty}\phi'(x)=0$ there exists $x_2(\epsilon)$ such that $x \geq x_2$ implies $\phi'(x) < \epsilon$, for all $\epsilon > 0$. Part $(\Rn{1})$ gives us the existence of a $T_3(\epsilon)$ such that $x(t) \geq x_2(\epsilon)$ whenever $t \geq T_3(\epsilon)$.
If we consider $t \geq T^* := 2 T_1(\epsilon) + 2 \tau + 2 T_2(\epsilon) + 2 T_3(\epsilon)$ we may exploit asymptotic monotonicity once more to arrive at
\begin{align*}\label{1st_est}
x'(t) &\geq (1-\epsilon) \int_{[0,\tau]} \mu_1(ds) \phi(x(t-s)) 
+ (1-\epsilon) \int_{[0,T_2]}\mu_2(ds) \phi(x(t-s)) \\
&\geq (1-\epsilon)\, \mu_1\, \phi(x(t-\tau)) + (1-\epsilon)^2\, \mu_2\, \phi(x(t-T_2)).
\end{align*}
Therefore for $t\geq T^\ast$, we have 
\begin{equation} \label{1st_est}
\frac{x'(t)}{\phi(x(t))} \geq (1-\epsilon)\, \mu_1\, \frac{\phi(x(t-\tau))}{\phi(x(t))} 
+ (1-\epsilon)^2\, \mu_2\, \frac{\phi(x(t-T_2))}{\phi(x(t))}.
\end{equation}
In a moment, we will show for any $\theta>0$ that 
\begin{equation} \label{eq.phixtphixttheta1}
\text{For each $\theta>0$, }
\lim_{t\to\infty} \frac{\phi(x(t-\theta))}{\phi(x(t))}=1.
\end{equation}
Therefore from \eqref{1st_est} and applying \eqref{eq.phixtphixttheta1} twice (with $\theta=\tau$ and $\theta=T_2$), we get 
\[
\liminf_{t\to\infty} \frac{x'(t)}{\phi(x(t))}\geq (1-\epsilon)\, \mu_1\ + (1-\epsilon)^2\, \mu_2\geq (1-\epsilon)^2M.
\]
Letting $\epsilon\to 0^+$ gives
\[
\liminf_{t\to\infty} \frac{x'(t)}{\phi(x(t))}\geq M.
\]
Asymptotic integration and the asymptotic equivalence of $\Phi$ and $F$ 
yields 
\[
\liminf_{t \to \infty}F(x(t))/Mt \geq 1.
\]
Combining this with the corresponding limsup from Step 1 proves the theorem.  

It remains to return to the deferred proof of \eqref{eq.phixtphixttheta1}. Since $x(t-\theta)<x(t)$ for all $t\geq \theta$, and $\phi$ is increasing, we immediately have 
\[
\limsup_{t\to\infty} \frac{\phi(x(t-\theta))}{\phi(x(t))}\leq 1.
\] 
To get the corresponding liminf, we start by applying the Mean Value Theorem to the $C^1$ function $a:[\theta,\infty)\to\mathbb{R}$ defined by $a(t):=(\phi \circ x)(t - \theta)$ for $t\geq \theta$, to find a $\theta_t \in [0,\theta]$ such that 
\begin{align}\label{MVT_tau}
\phi(x(t)) = \phi(x(t-\theta)) + \phi'(x(t-\theta_t))x'(t - \theta_t)\theta, \quad t\geq \theta.
\end{align}
We have already shown above that 
\[
\limsup_{t\to\infty} \frac{x'(t)}{\phi(x(t))}\leq M.
\]
Therefore, there is $T_5>0$ such that $0<x'(t)<2M\phi(x(t))$ for every $t\geq T_5$. Let $T_6(\theta)=T_5+\theta$. Then for 
$t>T_6(\theta)$, since $\theta_t\in [0,\theta]$, we have $t-\theta_t>T_5>0$, and so
\[
x'(t - \theta_t)<2M\phi(x(t-\theta_t))\leq 2M\phi(x(t)).
\]
Therefore by this inequality and \eqref{MVT_tau}, we get
\[
\frac{\phi(x(t-\theta))}{\phi(x(t))} >1-\phi'(x(t-\theta_t)) 2M\theta \geq 1-2M\theta\sup_{s\in [t-\theta,t]} \phi'(x(s)), \quad t>T_6(\theta),
\] 
where we have used the fact that $\theta_t\in [0,t]$ to get the last inequality on the right--hand side. Finally, as $\phi'(x)\to 0$ as $x\to\infty$ and 
$x(s)\to\infty$ as $s\to\infty$, it follows that
\[
\lim_{t\to\infty} \sup_{s\in [t-\theta,t]} \phi'(x(s)) = 0,
\] 
so taking limits in the last inequality yields
\[
\liminf_{t\to\infty} \frac{\phi(x(t-\theta))}{\phi(x(t))} \geq 1.
\]
Together with the corresponding limsup, we arrive at \eqref{eq.phixtphixttheta1}, as required. 
\end{proof}
\begin{proof}[Proof of Corollary \ref{Minfinity}]
Note that the solution of \eqref{VIDE}, $v$, and the solution of \eqref{functional}, $x$, obey $x(t) \geq v(t)$ for all $t \geq 0$. Henceforth we will work with $v)$ for convenience and we also note that $v(t)>0$ for all $t \geq 0$ and $\lim_{t\to\infty}v(t)=\infty$. By hypothesis, for each arbitrary positive, real number $N$ there exists $T_1(N)>0$ such that $\int_{[0,T_1(N)]}\mu_2(ds) > N$, for all $t \geq T_1(N)$. Similarly, by \eqref{fasym}, for all $\epsilon \in (0,1)$ there exists $T_2(\epsilon)>0$ such that $f(x) > (1-\epsilon)\phi(x)$ for all $x \geq T_2$. Since $\lim_{t\to\infty}v(t)=\infty$ there exists $x(\epsilon)$ such that $v(t) > T_2(\epsilon)$ for all $t \geq x(\epsilon)$. Hence for all $t \geq T := \max(2T_1,2T_2)$ and any $\epsilon \in (0,1)$ we have
\begin{align*}
v'(t) &= \int_{[0,T]}\mu_2(ds)f(v(t-s)) + \int_{(T,t]}\mu_2(ds)f(v(t-s))\\
&\geq (1-\epsilon)\int_{[0,T]}\mu_2(ds)\phi(v(t-s))
\geq (1-\epsilon)N\phi(v(t-T)).
\end{align*}
Now we define the following comparison equation for each fixed $\epsilon$ and $N$ by
\begin{align*}
y_N'(t) = \frac{N(1-\epsilon)}{2}\phi(y_N(t-T)), \, t >T; \,\, y_N(t) = \frac{v(t)}{2}, \, t \in [0,T].
\end{align*} 
Since $\phi$ is monotonically increasing it is clear that we have $y_N(t) < v(t)$ for all $t \geq 0$. Then let $u_N(t):= y_N(t+T)$ for $t \geq -T$. For $t>0$, $t+T>T$ and hence we have 
\begin{align*}
u_N'(t) = y_N'(t+T) = \frac{N(1-\epsilon)}{2}\,\phi(y_N(t)) = \frac{N(1-\epsilon)}{2}\,\phi(u_N(t-T)).
\end{align*}
For $t \in [-T,0]$, $u_N(t) = y_N(t+T) = v(t+T)/2 := \psi_N(t)$. Thus we have the following delay differential equation for $u_N(t)$,
\[
u_N'(t) = \frac{N(1-\epsilon)}{2}\phi(u_N(t-T)), \quad t >0; \quad u_N(t) = \psi_N(t) > 0, \quad t \in [-T,0].
\]
Applying Theorem \ref{increasing_con} yields
$
\lim_{t\to\infty}F(u_N(t))/t = N(1-\epsilon)/2.
$
This in turn implies that $\lim_{t\to\infty}F(y_N(t+T))/t = N(1-\epsilon)/2$. Finally, since $F$ is increasing and $v(t)$ lies above our comparison solution $y_N$, we obtain 
\[
\frac{N(1-\epsilon)}{2} = \lim_{t\to\infty}\frac{F(y_N(t))}{t}\frac{t}{t-T} = \liminf_{t\to\infty}\frac{F(y_N(t))}{t} \leq \liminf_{t\to\infty}\frac{F(v(t))}{t}.
\]
We can now let $\epsilon \to 0^+$ and, since $N$ was arbitrary, we have proven that 
\[
\liminf_{t\to\infty}\frac{F(v(t))}{t} = \infty.
\]
From the opening remark of the proof we see that the same conclusion holds with $x$, the solution of \eqref{functional}, in place of $v$.
\end{proof}
Before giving the proof of Theorem~\ref{F_inv} we first establish the following useful technical result.
\begin{lemma}\label{F_inv_suff}
Suppose $f \in C(\mathbb{R}^+;(0,\infty))$ is asymptotically increasing and $F$ satisfies $\lim_{x \to \infty}F(x)=\infty$. If 
\eqref{L} holds, then for every $\epsilon>0$ sufficiently small there exists $T(\epsilon)>0$ such that 
\[
1 < \frac{F^{-1}((1+\epsilon)t)}{F^{-1}(t)} < \frac{1}{1-\frac{\epsilon(1+\epsilon)}{1-\epsilon}L}, \,\, t \geq T(\epsilon).
\]
\end{lemma}
\begin{proof}[Proof of Lemma \ref{F_inv_suff}]
Consider $u'(t) = f(u(t)),\,\, t>0$ with $u(0)=1$. Then we have $u(t) = F^{-1}(t), \,\, t\geq 0$ and $\lim_{t \to \infty}u(t)=\infty$. Hence, for all $t \geq T_1(\epsilon)$ we have $u(t)>x_1(\epsilon)$, where $x_1(\epsilon)$ is defined by
$
1-\epsilon < \tfrac{f(x)}{\phi(x)}<1+\epsilon, \,\, x \geq x_1(\epsilon),
$  
and $\phi$ is an increasing function. Thus for $t\geq T_1(\epsilon)$,
\begin{align*}
0 &< F^{-1}((1+\epsilon)t)-F^{-1}(t) = \int_t^{(1+\epsilon)t}u'(s)ds = \int_t^{(1+\epsilon)t}f(u(s))ds \\
&\leq (1+\epsilon)\int_t^{(1+\epsilon)t}\phi(u(s))ds \leq \epsilon\,(1+\epsilon)\,t\,(\phi\circ F^{-1})((1+\epsilon)t).
\end{align*}
Therefore
\begin{align}\label{est_F}
0 < 1 - \frac{F^{-1}(t)}{F^{-1}((1+\epsilon)t)} \leq \epsilon(1+\epsilon)t \frac{\phi(F^{-1}((1+\epsilon)t))}{F^{-1}((1+\epsilon)t)}, \,\, t \geq T_1(\epsilon).
\end{align}
Now letting $y_\epsilon(t) = F^{-1}((1+\epsilon)t)$, so $F(y_\epsilon(t))=(1+\epsilon)t$ and $y_\epsilon(t) = F^{-1}((1+\epsilon)t) > F^{-1}(t)>x_1(\epsilon)$. Hence 
\begin{align*}
(1+\epsilon)t \frac{\phi(F^{-1}((1+\epsilon)t))}{F^{-1}((1+\epsilon)t)} = \frac{F(y_\epsilon(t))\phi(y_\epsilon(t))}{y_\epsilon(t)} < \frac{F(y_\epsilon(t))f(y_\epsilon(t))}{(1-\epsilon)y_\epsilon(t)}.
\end{align*}
Thus \eqref{est_F} becomes
\[
0 < 1 - \frac{F^{-1}(t)}{F^{-1}((1+\epsilon)t)} \leq \frac{\epsilon F(y_\epsilon(t))f(y_\epsilon(t))}{(1-\epsilon)y_\epsilon(t)}, \,\, t \geq T_1(\epsilon).
\]
Now by \eqref{L} there exists $x_2(\epsilon)>0$ such that $f(x)F(x)/x < L(1+\epsilon)$ for all $x \geq x_2(\epsilon)$. Let $T_2(\epsilon)>0$ be such that $F^{-1}(t)>x_2(\epsilon)$, which implies $y_\epsilon(t)>x_2(\epsilon)$ for all $t \geq T_2(\epsilon)$. Therefore, letting $T(\epsilon) = 1+\max(T_1(\epsilon),T_2(\epsilon))$, 
\[
0 < 1 - \frac{F^{-1}(t)}{F^{-1}((1+\epsilon)t)} \leq \frac{\epsilon F(y_\epsilon(t))f(y_\epsilon(t))}{(1-\epsilon)y_\epsilon(t)} \leq \frac{\epsilon(1+\epsilon) L}{(1-\epsilon)}, \quad t \geq T(\epsilon).
\]  
Thus, choosing $\epsilon \in (0,1/4 \vee 3L/5)$, $1-\frac{\epsilon(1+\epsilon)}{1-\epsilon}L>0$, so we have
\[
0 < 1- \frac{\epsilon(1+\epsilon)}{1-\epsilon}L < \frac{F^{-1}(t)}{F^{-1}((1+\epsilon)t)}, \quad t \geq T(\epsilon).
\]
Thus
\[
\frac{F^{-1}((1+\epsilon)t)}{F^{-1}(t)} < \frac{1}{1-\frac{\epsilon(1+\epsilon)}{1-\epsilon}L}, \quad t \geq T(\epsilon),
\]
as claimed.
\end{proof}

We are now in position to give the proof of Theorem \ref{F_inv}, as promised.
\begin{proof}[Proof of Theorem \ref{F_inv}]
By hypothesis, we can apply Theorem~\ref{increasing_con} to give $\lim_{t\to\infty} x(t)=+\infty$ and $\lim_{t\to\infty}F(x(t))/Mt = 1$. 
The latter limit implies that for each $\epsilon \in (0,1)$ there exists $T(\epsilon)>0$ such that 
$1-\epsilon < F(x(t))/Mt < 1+\epsilon$ for all $t\geq T(\epsilon)$. Hence
\[
\frac{F^{-1}((1-\epsilon)Mt)}{F^{-1}(Mt)} < \frac{x(t)}{F^{-1}(Mt)} < \frac{F^{-1}((1+\epsilon)Mt)}{F^{-1}(Mt)}, \quad t \geq T(\epsilon).
\]
Since $f$ obeys \eqref{fasym}, it follows that $F(x)\to\infty$ as $x\to\infty$. Therefore, all the hypotheses of 
Lemma \ref{F_inv_suff} hold, so we can apply it to the right--hand member of the above inequality. Doing this, and then sending $\epsilon \to 0$ yields $\limsup_{t\to\infty}x(t)/F^{-1}(Mt)\leq 1$. The liminf is dealt with analogously. 
\end{proof}

\section{Proofs With Decreasing Nonlinearity} \label{sec.decreasing.results}
This section concentrates on results in which $f$ is asymptotic to a decreasing function, principally Theorem~\ref{decreasing_con}. In order to prove it, we find it useful to prepare some estimates concerning the functions 
\[
F(x)=\int_1^x \frac{1}{f(u)}\,du, \quad \Phi(x)=\int_1^x \frac{1}{\phi(u)}\,du,
\]
where $\phi$ is a decreasing function asymptotic to $f$.

Our first result shows, when $\phi$ is strictly decreasing, that $\Phi^{-1}$ preserves asymptotic behaviour under translation.
\begin{lemma}\label{Phi_inv_1}
Suppose that $\phi \in C^1(\mathbb{R}^+;\mathbb{R}^+)$ is strictly decreasing. Define the strictly increasing function $\Phi$ by 
\[
\Phi(x) = \int_1^x \frac{1}{\phi(u)}du, \quad x\geq1.
\]
Then, for each $A \in \mathbb{R}$ and $B \in (0,\infty)$, 
\[
\lim_{t \to \infty}\frac{\Phi^{-1}(A+ B t)}{\Phi^{-1}(B t)} = 1.
\] 
\end{lemma}
\begin{proof}[Proof of Lemma \ref{Phi_inv_1}]
By construction $\Phi^{-1}$ is a $C^1$, positive and strictly increasing function on $[0,\infty)$ and we can always consider it on $[0,\infty)$ by taking $t$ sufficiently large. We begin by noting that since $\Phi$ is the integral of a nondecreasing function it is convex. Therefore $\Phi^{-1}$ is a concave function and $\Phi^{-1}(0) = 1$. This means that $\Phi^{-1}$ is subadditive and taking $A>0$ we may write
\[
\Phi^{-1}(A+Bt) \leq \Phi^{-1}(A) + \Phi^{-1}(Bt).
\]
Hence 
$
\Phi^{-1}(A+Bt)/\Phi^{-1}(Bt) \leq 1 + \Phi^{-1}(A)/\Phi^{-1}(Bt)
$
and since $\lim_{t \to \infty}\Phi^{-1}(t)=\infty$ taking the limsup yields
\[
\limsup_{t \to \infty}\frac{\Phi^{-1}(A+Bt)}{\Phi^{-1}(Bt)} \leq 1, \quad A>0 .
\]
If $A<0$, by monotonicity, $\Phi^{-1}(A+Bt) < \Phi^{-1}(Bt)$ and we quickly obtain 
\[
\limsup_{t \to \infty}\frac{\Phi^{-1}(A+Bt)}{\Phi^{-1}(Bt)} \leq 1, \quad A \in \mathbb{R}.
\]
Given $A>0$, $\Phi^{-1}(A+Bt) > \Phi^{-1}(Bt)$ and we obtain
\[
\liminf_{t \to \infty}\frac{\Phi^{-1}(A+Bt)}{\Phi^{-1}(Bt)} \geq 1.
\]
If $A<0$ apply the Mean Value Theorem to the $C^1$ function $\Phi^{-1}$ to find a $\theta_t \in [A+Bt, Bt]$ such that
$
\Phi^{-1}(Bt) = \Phi^{-1}(A+Bt) - A\,(\phi\circ\Phi^{-1})(\theta_t).
$
Note that, for $t$ sufficiently large, we can guarantee $\theta_t>0$. Therefore
\[
\frac{\Phi^{-1}(A+Bt)}{\Phi^{-1}(Bt)} = 1 + \frac{A\, (\phi\circ\Phi^{-1})(\theta_t)}{\Phi^{-1}(Bt)},
\]
and hence by monotonicity of $\phi$ and $\Phi^{-1}$
\[
\frac{\Phi^{-1}(A+Bt)}{\Phi^{-1}(Bt)} \geq 1 + \frac{A\, (\phi\circ\Phi^{-1})(0)}{\Phi^{-1}(Bt)}.
\]
Now we can use that $\lim_{t \to \infty}\Phi^{-1}(t)=\infty$ to obtain
\[
\liminf_{t \to \infty}\frac{\Phi^{-1}(A+Bt)}{\Phi^{-1}(Bt)} \geq 1 + \lim_{t \to \infty}\frac{A\, \phi(\Phi^{-1}(0))}{\Phi^{-1}(Bt)} = 1, \quad A<0.
\]
Combining these limits gives the result for $A \in \mathbb{R}$ and any $B \in (0,\infty)$. 
\end{proof} 
\begin{lemma}\label{Phi_inv_2}
Suppose that $\phi \in C^1(\mathbb{R}^+;\mathbb{R}^+)$ is strictly decreasing. Then, with $\Phi^{-1}$ defined as in Lemma \ref{Phi_inv_1}, for all $\epsilon\in (0,1)$ we have
\[
\frac{\Phi^{-1}((1+\epsilon)t)}{\Phi^{-1}(t)} < \frac{1}{1-\epsilon}.
\]
\end{lemma}
\begin{proof}[Proof of Lemma \ref{Phi_inv_2}] Consider the differential equation defined by
\begin{align}\label{aux_ODE}
w'(t) = \phi(w(t)), \,\, t>0; \,\, w(0)=1.
\end{align}
We have that $w(t) = \Phi^{-1}(t), \,\, t \geq 0$ and hence
\begin{align*}
\frac{\Phi^{-1}((1+\epsilon)t)}{\Phi^{-1}(t)} &= \frac{w((1+\epsilon)t)}{w(t)} = \frac{w(t) + \int_{t}^{t+\epsilon t}w'(s)ds}{w(t)}
= 1 + \frac{1}{w(t)}\int_t^{t + \epsilon t}\phi(w(s))ds. 
\end{align*}
Now using the monotonicity of both the solution and of $\phi$ we have
\[
\frac{\Phi^{-1}((1+\epsilon)t)}{\Phi^{-1}(t)} \leq 1 + \frac{\epsilon t \phi(w(t))}{w(t)} = 1 + \epsilon t \frac{\phi(\Phi^{-1}(t))}{\Phi^{-1}(t)}.
\] 
For $t\geq 0$, by setting $y:= \Phi^{-1}(t)\geq 1$, we obtain
\[
\frac{t \phi(\Phi^{-1}(t))}{\Phi^{-1}(t)} = \frac{\Phi(y)\phi(y)}{y} = \frac{\phi(y)}{y}\int_1^y \frac{1}{\phi(u)}du \leq \frac{y-1}{1-\epsilon}\frac{1}{\phi(y)}\frac{\phi(y)}{y} \leq \frac{1}{1-\epsilon}.
\]
Combining these estimates yields
\[
\frac{\Phi^{-1}((1+\epsilon)t)}{\Phi^{-1}(t)} \leq 1 + \frac{\epsilon}{1-\epsilon} \leq \frac{1}{1-\epsilon},
\]
as required.
\end{proof}
\begin{lemma} \label{FinvPhiinv}
Suppose that $f\in C(\mathbb{R}^+;(0,\infty))$ and $f$ is asymptotic to the $C^1$ decreasing function $\phi$. 
Let $F$ be given by \eqref{cap_F} and $\Phi$ be defined as in Lemma \ref{Phi_inv_1}. Then 
\[
\lim_{t\to\infty} \frac{F^{-1}(t)}{\Phi^{-1}(t)}=1.
\]
\end{lemma}
\begin{proof}[Proof of Lemma~\ref{FinvPhiinv}]
Notice that the solution $u$ of 
\begin{align}\label{aux_ODE2}
u'(t) = f(u(t)), \quad t>0; \quad u(0)=1
\end{align}
is $u(t)=F^{-1}(t)$ for $t\geq 0$. For every $\epsilon\in (0,1/2)$ there is $x_1(\epsilon)>0$ such that 
$1-\epsilon<f(x)/\phi(x)<1+\epsilon$ for all $x>x_1(\epsilon)$. Since $u(t)\to\infty$ as $t\to\infty$, it follows that there exists
$T(\epsilon)>0$ such that $u(t)>x_1(\epsilon)$ for all $t\geq T(\epsilon)$. Hence 
\[
   u'(t)=f(u(t))\in ((1-\epsilon)\phi(u(t)),(1+\epsilon)\phi(u(t))), \quad t\geq T(\epsilon).
\]
Hence 
\[
1-\epsilon<
\frac{u'(t)}{\phi(u(t))}< 1+\epsilon, \quad t\geq T(\epsilon).
\]
and integration over $[T(\epsilon),t]$ yields, with $\Phi^\ast:=\Phi(x(T(\epsilon)))$,
\begin{equation*}
\Phi^\ast+
(1-\epsilon)(t-T(\epsilon)) < \Phi(u(t)) < \Phi^\ast+(1+\epsilon)(t-T(\epsilon)), \quad t\geq T(\epsilon),
\end{equation*}
and recalling that $u(t)=F^{-1}(t)$, we have
\begin{equation}\label{eq.Phiut}
\Phi^{-1}(\Phi^\ast+
(1-\epsilon)(t-T(\epsilon)) )< F^{-1}(t) < \Phi^{-1}(\Phi^\ast+(1+\epsilon)(t-T(\epsilon))), \quad t\geq T(\epsilon).
\end{equation}
Applying Lemma \ref{Phi_inv_1} to the left and right--hand sides of \eqref{eq.Phiut} shows that 
\begin{gather*}
\liminf_{t\to\infty} \frac{\Phi^{-1}((1-\epsilon)t)}{\Phi^{-1}(t)}
\leq \liminf_{t\to\infty} \frac{F^{-1}(t)}{\Phi^{-1}(t)} \leq
\limsup_{t\to\infty} \frac{F^{-1}(t)}{\Phi^{-1}(t)} \leq \limsup_{t\to\infty} \frac{\Phi^{-1}((1+\epsilon)t)}{\Phi^{-1}(t)}. 
\end{gather*}
By Lemma~\ref{Phi_inv_2}, we have 
\[
\Phi^{-1}((1+\epsilon)t)<\frac{1}{1-\epsilon} \Phi^{-1}(t),
\]
so immediately we see that 
\[
\limsup_{t\to\infty} \frac{F^{-1}(t)}{\Phi^{-1}(t)} \leq \frac{1}{1-\epsilon},
\]
and letting $\epsilon\to 0^+$ gives
\begin{equation} \label{eq.limsupFPhiinv1}
\limsup_{t\to\infty} \frac{F^{-1}(t)}{\Phi^{-1}(t)} \leq 1.
\end{equation}
To deal with the liminf, write $y:=(1-\epsilon)t$ and $\eta:=(1-\epsilon)^{-1}-1$. Note that $\epsilon<1/2$ yields $\eta\in (0,1)$. Then by Lemma~\ref{Phi_inv_2} (with $\eta$ in the role of $\epsilon$), we get
\[
\frac{\Phi^{-1}((1-\epsilon)t)}{\Phi^{-1}(t)} = \frac{\Phi^{-1}(y)}{\Phi^{-1}((1+\eta)y)}>1-\eta=2-\frac{1}{1-\epsilon}.
\]
Hence 
\[
\liminf_{t\to\infty} \frac{F^{-1}(t)}{\Phi^{-1}(t)} \geq 2-\frac{1}{1-\epsilon}.
\]
Letting $\epsilon\to 0^+$ and combining the resulting inequality with \eqref{eq.limsupFPhiinv1} yields the desired limit.
\end{proof}

\begin{proof}[Proof of Theorem \ref{decreasing_con}]
The proof of the first claim is as before since no hypothesis regarding the asymptotic monotonicity of $f$ was used to show that the solution tends to $\infty$ as $t\to\infty$.

\emph{Step 1:} We first establish the required lower bound on the rate of growth of the solution. If $\epsilon>0$ is arbitrary, by hypothesis, there exists $x_1(\epsilon)$ such that for all $x > x_1(\epsilon)$, 
$
(1-\epsilon)\phi(x) < f(x) < (1+\epsilon)\phi(x).
$
Hence by part (\Rn{1}) there exists $T_1(\epsilon)$ such that for $t \geq T_1(\epsilon)$, $x(t) > x_1(\epsilon)$. Now let $T = T_1 + \tau + T_2$, where
\[
\int_{[0,T_2]} \mu_2(ds) > (1-\epsilon) \int_{[0,\infty)} \mu_2(ds), \quad t \geq T_2.
\]
Let $t \geq T(\epsilon)$, then $t-\tau \geq T_1$ and $x(t-s) > x_1(\epsilon)$ for $s \in [0,\tau]$. Hence 
$
f(x(t-s)) < (1+\epsilon)\phi(x(t-s)) < (1+\epsilon)\phi(x(t-\tau)).
$
Therefore
\begin{align}\label{est_1}
\int_{[0,\tau]} \mu_1(ds) f(x(t-s)) < (1+\epsilon)\int_{[0,\tau]} \mu_1(ds) \phi(x(t-\tau)), \quad t \geq T.
\end{align}
For $t \geq T(\epsilon)$ and $s \in [0,\tau]$,  
$
f(x(t-s)) \geq (1-\epsilon) \phi(x(t-s)) \geq (1-\epsilon)\phi(x(t)).
$
Thus
\begin{align}\label{est_2}
\int_{[0,\tau]} \mu_1(ds) f(x(t-s)) \geq (1-\epsilon) \int_{[0,\tau]} \mu_1(ds) \phi(x(t)).
\end{align}
Also, for $t \geq 2T$,
\[
\int_{[0,t]} \mu_2(ds) f(x(t-s)) \geq (1-\epsilon)\int_{[0,T]} \mu_2(ds)\phi(x(t-s)) 
\geq (1-\epsilon)\int_{[0,T]} \mu_2(ds) \phi(x(t)). 
\]
These estimates give us
\[
x'(t) \geq \left[(1-\epsilon)\int_{[0,\tau]} \mu_1(ds) + (1-\epsilon)\int_{[0,T]} \mu_2(ds) \right]\phi(x(t)) := M_\epsilon \phi(x(t)), \,\, t \geq 2T.
\]
Hence, defining $\Phi(x)$ as before and $\Phi_\epsilon := \Phi(x(2T))$, it can be shown by integration and rearrangement that
\begin{align}\label{est_3}
x(t) \geq \Phi^{-1}(\Phi_\epsilon + M_\epsilon(t-2T)), \quad t \geq 2T.
\end{align}
Hence
\[
\liminf_{t \to \infty}\frac{x(t)}{\Phi^{-1}(M_\epsilon t)} \geq \liminf_{t \to \infty}\frac{\Phi^{-1}(M_\epsilon t +\Phi_{T,\epsilon})}{\Phi^{-1}(M_\epsilon t)},
\]
where $\Phi_{T,\epsilon} = \Phi_\epsilon - 2T M_\epsilon$. Now for each fixed $\epsilon > 0$ we may apply Lemma \ref{Phi_inv_1} to obtain
$
\liminf_{t \to \infty}x(t)/\Phi^{-1}(M_\epsilon t) \geq 1.
$
We first note that by construction we have the inequalities 
$
(1-\epsilon)^2 M < M_\epsilon < M,
$
and thus $M_\epsilon \to M$ as $\epsilon \downarrow 0$.
Now consider
\[
\liminf_{t \to \infty}\frac{x(t)}{\Phi^{-1}(Mt)} = \liminf_{t \to \infty}\frac{x(t)}{\Phi^{-1}(M_\epsilon t)}\frac{\Phi^{-1}(M_\epsilon t)}{\Phi^{-1}(Mt)}.
\]
Letting $\theta = M_\epsilon t$ and $\lambda_\epsilon = M / M_\epsilon > 1$ we have
\[
\frac{\Phi^{-1}(M_\epsilon t)}{\Phi^{-1}(Mt)} = \frac{\Phi^{-1}(\theta)}{\Phi^{-1}(\lambda_\epsilon \theta)}> 2 - \lambda_\epsilon,
\]
where the final inequality is obtained using the estimate from Lemma \ref{Phi_inv_2} with $\lambda_\epsilon = 1 + \epsilon$.
As $\epsilon \downarrow 0, \,\, \lambda_\epsilon \to 1$ and hence we conclude that
$
\liminf_{t \to \infty}x(t)/\Phi^{-1}(Mt) \geq 1.
$
The asymptotic equivalence of $f$ and $\phi$ then yields
\[
\liminf_{t \to \infty}\frac{x(t)}{F^{-1}(Mt)} \geq 1.
\]
\emph{Step 2:} We now proceed to supply the corresponding upper bound. $\epsilon>0$ is once again arbitrary and using \eqref{est_1} we have
\begin{align*}
x'(t) &< (1+\epsilon) \int_{[0,\tau]} \mu_1(ds) \phi(x(t-\tau)) 
+ \int_{[0,t-2T]} \mu_2(ds) f(x(t-s)) \\
&\qquad+ \int_{(t-2T,t]} \mu_2(ds) f(x(t-s)), \quad t \geq 2T.
\end{align*}
Now using the monotonicity of $\phi$ and \eqref{est_3} we arrive at
\begin{align}\label{est_4}
x'(t) &\leq (1+\epsilon) \int_{[0,\tau]} \mu_1(ds) \phi(x(t-\tau)) + (1+\epsilon)\int_{[0,t-2T]} \mu_2(ds) \phi(x(t-s)) \nonumber\\
&\qquad + \int_{(t-2T,t]} \mu_2(ds) f(x(t-s)) \nonumber \\
&\leq (1+\epsilon) \int_{[0,\tau]} \mu_1(ds) (\phi \circ \Phi^{-1})(\Phi_\epsilon + M_\epsilon(t - \tau - 2T)) \nonumber\\
&\qquad+ \int_{(t-2T,t]} \mu_2(ds) f(x(t-s)) \nonumber \\ &\qquad\qquad+ (1+\epsilon) \int_{[0,t-2T]} \mu_2(ds) (\phi \circ \Phi^{-1})(\Phi_\epsilon + M_\epsilon(t - s - 2T)) \nonumber \\
&=: a_1(t) + a_2(t) + a_3(t), \quad t \geq 3T.\end{align}
Hence integration yields 
$
x(t) \leq x(3T) + \int_{3T}^t \{a_1(s) + a_2(s) + a_3(s)\}\,ds.
$
We estimate the first term as follows
\begin{align*}
\int_{3T}^t a_1(s)ds &= (1+\epsilon)\int_{[0,\tau]} \mu_1(ds) \int_{3T}^t (\phi \circ \Phi^{-1})(\Phi_\epsilon + M_\epsilon(r-\tau-2T))\,dr \\
&= \frac{(1+\epsilon)\mu_1}{M_\epsilon}\int_{\Phi_\epsilon + M_\epsilon(T-\tau)}^{\Phi_\epsilon + M_\epsilon(t-\tau-2T)}(\phi \circ \Phi^{-1})(u)\,du \\
&=\frac{(1+\epsilon)\mu_1}{M_\epsilon}\int_{\Phi^{-1}(\Phi_\epsilon + M_\epsilon(T-\tau))}^{\Phi^{-1}(\Phi_\epsilon + M_\epsilon(t-\tau-2T))} dv \\
&\leq \frac{(1+\epsilon)\mu_1}{M_\epsilon} \left[\Phi^{-1}(\Phi_\epsilon + M_\epsilon(t-\tau-2T)) \right], \quad t \geq 3T.
\end{align*}
The second term can be estimated as follows
\[
a_2(t) = \int_{(t - 2T,t]}\mu_2(ds)f(x(t-s)) \leq \int_{(t-2T,t]} \mu_2(ds) \cdot \sup_{u \in[0,2T]}f(x(u)).
\]
Integrating and changing the order of integration then yields
\begin{align*}
\int_{3T}^t a_2(s)\,ds &\leq \int_{3T}^t\int_{(s-2T,s]}\mu_2\,(dr)ds \sup_{u \in[0,2T]}f(x(u)) \\
&=\int_{[T,t]} \{t \wedge(2T+r) - (3T \vee r)\}\mu_2(dr) \sup_{u \in[0,2T]}f(x(u)).
\end{align*}
We then take cases and find that this estimate can be reduced to 
\[
\int_{3T}^t a_2(s)\,ds \leq 2T \mu_2 \sup_{u \in[0,2T]}f(x(u)) := A_T, \quad t \geq 3T.
\]
The last term is then estimated as follows
\begin{align*}
\int_{3T}^t a_3(s)ds &= (1+\epsilon)\int_{3T}^t\int_{[0,s-2T]}\mu_2(du)(\phi \circ \Phi^{-1})(\Phi_\epsilon+ M_\epsilon(s-u-2T))\,ds \\
&= (1+\epsilon)\int_{T}^{t-2T}\int_{[0,w]}\mu_2(du)(\phi \circ \Phi^{-1})(\Phi_\epsilon+ M_\epsilon(w-u))\,dw \\
&= (1+\epsilon)\int_{[0,t-2T]}\mu_2(dw) \int_{u \wedge T}^{t-2T}(\phi \circ \Phi^{-1})(\Phi_\epsilon+ M_\epsilon(w-u))\,du \\
&= (1+\epsilon)\int_{[0,T]}\mu_2(dw) \int_{T-u}^{t-2T-u}(\phi \circ \Phi^{-1})(\Phi_\epsilon+ M_\epsilon s)\,ds \\
&+ (1+\epsilon)\int_{(T,t-2T]}\mu_2(dw) \int_{0}^{t-2T-u}(\phi \circ \Phi^{-1})(\Phi_\epsilon+ M_\epsilon s)\,ds \\
&\leq (1+\epsilon)\mu_2 \int_0^{t-2T}(\phi \circ \Phi^{-1})(\Phi_\epsilon + M_\epsilon s)\,ds.
\end{align*}
Finally, a rearrangement of the type performed in the calculation of $\int_{3T}^t a_1(s)ds$, simplifies this estimate to
\[
\int_{3T}^t a_3(s)ds \leq \frac{(1+\epsilon)\mu_2}{M_\epsilon}\left[ \Phi^{-1}(\Phi_\epsilon + M_\epsilon(t-2T)) \right], \quad t \geq 3T.
\]
Combining these three estimates we arrive at
\[
x(t) \leq x(3T) + A_T + \frac{(1+\epsilon)(\mu_1 + \mu_2)}{M_\epsilon}\Phi^{-1}(\Phi_\epsilon + M_\epsilon(t-2T)), \quad t \geq 3T.
\]
Hence
\begin{align*}
\limsup_{t \to \infty}\frac{x(t)}{\Phi^{-1}(Mt)} &\leq  \frac{(1+\epsilon)M}{M_\epsilon}\limsup_{t \to \infty}\frac{\Phi^{-1}(\Phi_\epsilon + M_\epsilon(t-2T))}{\Phi^{-1}(Mt)}.
\end{align*}
We now note that the arguments for the limsup in Lemma \ref{Phi_inv_1} work for the limsup above since $M_\epsilon<M$ and hence
$
\limsup_{t \to \infty}x(t)/\Phi^{-1}(Mt) \leq  (1+\epsilon)M/M_\epsilon.
$
Therefore we may send $\epsilon \downarrow 0$ and the same arguments as before yield
\[
\limsup_{t \to \infty} \frac{x(t)}{F^{-1}(Mt)} \leq 1.
\]
Combining this with our lower bound gives the desired conclusion.
\end{proof}
\section{Proofs of Results with Regularly Varying Nonlinearity} \label{sec.regvarproofs}
Our final section of proofs concern results in which it is assumed that $f$ is regularly varying at infinity.
\begin{proof}[Proof of Theorem \ref{RV_0}]
As before $z(t) \to \infty$ as $t \to \infty$. From \eqref{DDE} we have
\[
\frac{z'(t)}{z(t)} = \int_{[0,\tau]} \mu_1(ds) \frac{f(z(t-s))}{z(t-s)}\frac{z(t-s)}{z(t)} 
\leq \int_{[0,\tau]} \mu_1(ds) \frac{f(z(t-s))}{z(t-s)},
\]
where we have used the fact that $z(t)$ is strictly increasing for $t$ sufficiently large. Since $\lim_{x \to \infty}f(x)/x = 0$ there exists $x_1(\epsilon)$ such that for all $x>x_1(\epsilon)$ we have $f(x)/x < \epsilon$, for some arbitrary $\epsilon \in (0,1/2)$. Similarly, there exists $T(\epsilon)$ such that for all $t \geq T(\epsilon)$, $z(t) > x_1(\epsilon)$. Hence, for all $t \geq T^* := T(\epsilon)+\tau$, we have
\[
\frac{z'(t)}{z(t)} \leq \epsilon \int_{[0,\tau]} \mu_1(ds).
\]
Therefore $\limsup_{t \to \infty}z'(t)/z(t) = 0$. But since $z'(t)>0$ for $t$ sufficiently large and $z(t)>0$ for all $t \geq 0$ we also have $\liminf_{t \to \infty}z'(t)/z(t)\geq 0$ and hence $\lim_{t \to \infty}z'(t)/z(t) = 0$. Writing this as an $\epsilon-\delta$ statement it quickly follows that
\[
\lim_{t \to \infty}\sup_{s \in [0, T_1(\epsilon)]}\left|\frac{z(t-s)}{z(t)} -1 \right|=0,
\]
for some $T_1(\epsilon)$. Thus there exists $T_2(\eta,\epsilon)$ such that for all $t \geq T_2(\eta,\epsilon)$ we have
\[
\sup_{s \in [0,T_2(\eta,\epsilon)]}\left|\frac{z(t-s)}{z(t)} -1 \right| < \eta.
\] 
Therefore 
$
1-\eta < z(t-s)/z(t) < 1, \,\, s \in [0,T_1(\epsilon)], \,\, t \geq T_2(\eta,\epsilon).
$
Taking $\eta = \epsilon$ we have 
$
\lambda_{t,s} := z(t-s)/z(t) \in [1-\epsilon,1],
$
for all $s \in [0,T_1(\epsilon)]$ and $t \geq T_2$. Thus for all $t \geq T_2$
\begin{align*}
\sup_{s \in [0,T_1]}\left|\frac{f(z(t-s))}{f(z(t))} - 1 \right| &= \sup_{s \in [0,T_1]}\left|\frac{f(\lambda_{t,s}z(t))}{f(z(t))} - 1 \right| \\
&\leq \sup_{\lambda \in [1-\epsilon,1]}\left|\frac{f(\lambda z(t))}{f(z(t))} - 1 \right|
\leq \sup_{\lambda \in [0,1/2]}\left|\frac{f(\lambda z(t))}{f(z(t))} - 1 \right|.
\end{align*} 
Since $f \in \text{RV}_\infty(0)$ the Uniform Convergence Theorem for slowly varying functions (cf. \cite[Theorem 1.2.1]{BGT}) allows us to now conclude that 
\[
\lim_{t \to \infty}\sup_{\lambda \in [0,1/2]}\left|\frac{f(\lambda z(t))}{f(z(t))} - 1 \right|=0,
\]
and hence
\begin{align}\label{unilim}
\lim_{t \to \infty}\sup_{s \in [0,T_1]}\left|\frac{f(z(t-s))}{f(z(t))} - 1 \right|=0.
\end{align}
Returning to \eqref{DDE} we now use this to estimate
\[
\frac{z'(t)}{f(z(t))} = \int_{[0,\tau]} \mu_1(ds)\frac{f(z(t-s))}{f(z(t))}.
\]
From \eqref{unilim} we have that
$
1-\epsilon < f(z(t-s))/f(z(t)) < 1+\epsilon,\,\, s \in [0,\tau], \,\, t \geq T_3(\epsilon).
$
Thus for all $t \geq T_3(\epsilon)$
\[
(1-\epsilon)M < \frac{z'(t)}{f(z(t))} < (1+\epsilon)M,
\]
and hence $\lim_{t \to \infty}z'(t)/f(z(t)) = M$. Once more asymptotic integration yields 
\[
\lim_{t \to \infty}\frac{F(z(t))}{Mt}=1.
\]
Therefore for all $\epsilon\in (0,1)$ there exists $T(\epsilon)$ such that for all $t \geq T(\epsilon)$, $Mt(1-\epsilon)<F(z(t))<Mt(1+\epsilon)$. Hence
\[
\frac{F^{-1}(Mt(1-\epsilon))}{F^{-1}(Mt)} < \frac{z(t)}{F^{-1}(Mt)} < \frac{F^{-1}(Mt(1+\epsilon))}{F^{-1}(Mt)}, \quad t \geq T(\epsilon).
\]
Now we note that since $F^{-1}\in \text{RV}_\infty(1)$ sending $t \to \infty$ yields
\[
1-\epsilon\leq \liminf_{t\to\infty}\frac{z(t)}{F^{-1}(Mt)}\leq \limsup_{t\to\infty} \frac{z(t)}{F^{-1}(Mt)}\leq 1+\epsilon. 
\]
Sending $\epsilon \to 0^+$ then gives the result.
\end{proof}
\begin{proof}[Proof of Theorem \ref{RV_0_2}]
By hypothesis there exist positive real numbers $\underline{f}$ and $\bar{f}$ such that $\underline{f} < f(x)< \bar{f}$ for all $x>0$. Hence, for any $t \geq T_2>0$, we have
\begin{align*}
v'(t) &= \int_{[0,T_2]}\mu_2(ds)f(v(t-s))+\int_{(T_2,t]}\mu_2(ds)f(v(t-s))\\
&\leq \int_{[0,T_2]}\mu_2(ds)f(v(t-s))+ \int_{(T_2,t]}\mu_2(ds)\bar{f}.
\end{align*}
Since $f(v(t))>\underline{f}>0$ for all $t \geq 0$, we have
\[
\frac{v'(t)}{f(v(t))} \leq \frac{\bar{f}}{\underline{f}}\int_{[T_2,\infty)}\mu_2(ds) 
+ \frac{1}{f(v(t))}\int_{[0,T_2]}\mu_2(ds)f(v(t-s)).
\]
By the arguments from Theorem \ref{RV_0}, $\lim_{t \to \infty}v'(t)/v(t)=0$ and 
\[
\lim_{t \to \infty}\sup_{0 \leq s \leq T_2}\left| \frac{f(v(t-s))}{f(v(t))}-1 \right|=0.
\]
Hence there exists $T_3(\epsilon)$ such that for all $t \geq T_3(\epsilon)$, $f(v(t-s))/f(v(t))<1+\epsilon$, for all $s \in [0,T_2(\epsilon)]$. Now if $T_2(\epsilon)$ is large enough that $\int_{[t,\infty)}\mu(ds)<\epsilon$ for all $t \geq T_2$, take $T:= T_2(\epsilon)+T_3(\epsilon)$. Thus for all $t \geq T$,
\begin{align*}
\frac{v'(t)}{f(v(t))} &< \epsilon\, \frac{\bar{f}}{\underline{f}}+ (1+\epsilon)\int_{[0,T_2]}\mu_2(ds) 
\leq \epsilon\, \frac{\bar{f}}{\underline{f}}+ (1+\epsilon)M.
\end{align*}
Therefore letting $t \to \infty$ and $\epsilon \to 0^+$ yields $\limsup_{t \to \infty}v'(t)/f(v(t)) \leq M$. The usual considerations and using that $F^{-1}\in \text{RV}_\infty(1)$ then give us the upper bound $\limsup_{t \to \infty}v(t)/F^{-1}(Mt)\leq 1.$ We defer the calculation of the required lower bound to the next Theorem.
\end{proof}
\begin{proof}[Proof of Theorem \ref{RV_0_lower_bound}]
Monotonicity of the solution means that for all $s \in [0,t]$, $v(t-s)/v(t) \leq 1$. Hence, with $T \in [0,t]$ arbitrary, we have
\begin{align*}
\frac{v'(t)}{v(t)} &= \int_{[0,t]} \mu_2(ds) \frac{f(v(t-s))}{v(t-s)}\frac{v(t-s)}{v(t)} \\
&\leq \int_{[0,t-T]}\mu_2(ds) \frac{f(v(t-s))}{v(t-s)} + \int_{(t-T,t]} \mu_2(ds) \frac{f(v(t-s))}{v(t-s)}.
\end{align*}
Now noting that $f(x)/x \to 0$ as $x \to \infty$ and that the measure is finite we may conclude that $\limsup_{t \to \infty}v'(t)/v(t)=0$. Therefore $\lim_{t \to \infty}v'(t)/v(t)=0$. As in the previous Corollary this quickly yields
\[
\lim_{t \to \infty}\sup_{s \in [0, T_1(\epsilon)]}\left|\frac{v(t-s)}{v(t)} -1 \right|=0.
\]
If $\epsilon \in (0,1)$ be arbitrary, by hypothesis, $\int_{[0,T(\epsilon)]}\mu(ds) \geq (1-\epsilon)M$ for some $T(\epsilon)$. Hence
\[
v'(t) \geq \int_{[0,T(\epsilon)]}\mu_2(ds) f(v(t-s)).
\]
As before regular variation can be used to show that 
\[
\lim_{t \to \infty} \sup_{s \in [0,T(\epsilon)]}\left| \frac{f(v(t-s))}{f(v(t))} - 1 \right| = 0.
\]
Therefore we have that
$
v'(t) \geq (1-\epsilon)^2 M f(v(t))
$
for all $t \geq T^*(\epsilon)$, for some $T^*(\epsilon)$. At this point the usual calculation reveals that
$
\liminf_{t \to \infty}F(v(t))/Mt \geq 1.
$
As before the fact that we have $F^{-1} \in \text{RV}_\infty(1)$ means that we can quickly improve this to 
\[
\liminf_{t \to \infty}\frac{v(t)}{F^{-1}(Mt)} \geq 1,
\]
completing the proof.
\end{proof}
\begin{proof}[Proof of Theorem \ref{RV_0_upper}]
Since $f\in \text{RV}_\infty(0)$ we have $\lim_{x\to\infty}f(x)/x=0$. Therefore there exists $X(\epsilon)$ such that for all $\epsilon>0$, $x^{-\epsilon}<f(x)<x^\epsilon$ for all  $x>X(\epsilon)$. SInce $\lim_{t\to\infty}v(t)=\infty$ there exists $T(\epsilon)$ such that for all $t \geq T(\epsilon)$ we have $v(t) > X(\epsilon)$ and hence 
\begin{align*}
v'(t) &= \int_{[0,t-T]}\mu_2(ds)f(v(t-s)) + \int_{(t-T,t]} \mu_2(ds)f(v(t-s)) \\
&\leq \int_{[0,t-T]}\mu_2(ds)v(t-s)^\epsilon + \int_{(t-T,t]} \mu_2(ds)f(v(t-s)).
\end{align*}
Letting $h(t)=\int_{(t-T,t]} \mu_2(ds)f(v(t-s))$, $h(t)\leq \int_{[t-T,t]} \mu_2(ds)\sup_{s\in [0,T]}f(v(s))$. Thus
$v'(t) \leq M v(t)^\epsilon + h(t), \,\, t \geq T(\epsilon)$. Therefore
\[
\frac{v'(t)}{v(t)^\epsilon} \leq M + \frac{h(t)}{v(t)^\epsilon}, \quad t \geq T(\epsilon).
\]
Since $h(t)\to 0$ as $t \to \infty$ taking the limsup yields 
$\limsup_{t\to\infty}v'(t)/v(t)^\epsilon \leq M$. Asymptotic integration of this inequality gives us 
\[
v(t)^{1-\epsilon} \leq (1-\epsilon)M(1+\epsilon)(t-T_1) + v(T_1)^{1-\epsilon}, \quad t \geq T_1,
\]
for an arbitrary $\epsilon \in (0,1)$ and some $T_1(\epsilon) > 0$. Taking logarithms and sending $t\to\infty$ and then $\epsilon \to 0$ we obtain $\limsup_{t\to\infty}\log(v(t))/\log(t) \leq 1$. $f\in \text{RV}_\infty(0)$ implies that $F \in \text{RV}_\infty(1)$ and hence $\lim_{x\to\infty}\log(F(x))/\log(x)=1$. Using the lower bound from Theorem \ref{RV_0_lower_bound} there exists $T_2$ such that for all $\epsilon \in (0,1)$ we have $v(t) > F^{-1}(M(1-\epsilon)t),\,\, t \geq T_2$. Similarly, $\log(v(t))/\log(t) \geq \log(F^{-1}(M(1-\epsilon)t))/\log(t)$. Taking the liminf then gives us that
\[
\liminf_{t\to\infty}\frac{\log(v(t))}{\log(t)} \geq \liminf_{t\to\infty}\frac{\log(F^{-1}(t))}{\log(t)} = 1.
\]
Combining the upper and lower bounds gives the desired result.
\end{proof}
\section{Justification of Examples}\label{sec_exp}
In this section we provide the relevant details to support the examples discussed in Section \ref{sec.examples}. The calculations for both examples are identical except for the final few steps where differing hypotheses are employed.
We begin by stating some formulae which are derived by integrating \eqref{f_prime}.
For $n \in \mathbb{N}$ and $x \in [n, n+ w_n/2)$
\begin{align}\label{f_1}
f(x) = f(n) + (x-n)\eta(n) + \frac{h_n - \eta(n)}{w_n}(x-n)^2.
\end{align} 
Hence $f(n + w_n/2) = f(n) + w_n \eta(n)/4 + (h_n w_n)/4$.
For $n \in \mathbb{N}$ and $x \in (n+ w_n/2,\, w_n+n]$
\begin{align}\label{f_2}
f(x) = f(n+ \tfrac{w_n}{2}) + h_n \left(x-n-\tfrac{w_n}{2}\right) + \frac{\eta(n+w_n) - h_n}{w_n}\left(x-n - \tfrac{w_n}{2}\right)^2.
\end{align}
Therefore $f(n+w_n) = f(n+ w_n/2) + (h_n w_n)/2 + (w_n/4)\left(\eta(n)+\eta(n+w_n)\right)$ . Finally for $x \in (n+w_n,n+1)$
\begin{align}\label{f_3}
f(x) = f(n+w_n) + \int_{n+w_n}^x \eta(u)du.
\end{align}
It follows that 
\begin{align}\label{f_n+1}
f(n+1) = f(n) + \frac{h_n w_n}{2} + \frac{w_n}{4}\left(\eta(n)+\eta(n+w_n)\right) + \int_{n+w_n}^x \eta(u)du.
\end{align}
Hence it can be shown that 
\begin{align}\label{f_sum_form}
f(n+1) = f(n) + \sum_{j=1}^n \left\{\frac{h_j w_j}{2} + \frac{w_j}{4}\left(\eta(j)+\eta(j+w_j)\right) + \int_{w_j+j}^{j+1}\eta(u)du \right\}.
\end{align}
\subsection{Example \ref{first_example}}
By hypothesis $\phi$ grows more quickly than the sums of $\eta(j)w_j$ and $h_j w_j$ so we only need to study the asympotics of the final term of \eqref{f_sum_form}. For $n \in \mathbb{N}$, 
\[
S_n := \sum_{j=1}^n \int_{j+w_j}^{j+1}\eta(u)du \leq \sum_{j=1}^n \int_{j}^{j+1}\eta(u)du.
\]
Thus $S_n \leq \int_1^{n+1}\eta(u)du$. \eqref{f_n+1} can be rewritten as
\begin{align}\label{f_n_ineq}
f(n+1) \leq \int_0^{n+1}\eta(u)du + T_n = \phi(n+1) +T_n,
\end{align}
 where $T_n/\phi(n) \to 0$ as $n\to\infty$. Similarly,
\[
S_n = \int_1^{n+1}\eta(u)du - \sum_{j=1}^n \int_j^{j+w_j}\eta(u)du.
\]
Since $\eta$ is decreasing $\sum_{j=1}^n \int_j^{j+w_j}\eta(u)du \leq \sum_{j=1}^n w_j \eta(j)$ and we have the estimate
\[
S_n \geq \int_1^{n+1}\eta(u)du - \sum_{j=1}^n w_j \eta(j).
\]
Hence \eqref{f_n+1} becomes
\[
f(n+1) \geq \phi(n+1)- \sum_{j=1}^n w_j \eta(j) + T_n,
\]
where $T_n/\phi(n) \to 0$ as $n\to\infty$. Combining our upper and lower estimates for $f(n+1)$ yields $\lim_{n\to\infty}f(n+1)/\phi(n+1)=1$. Since $\lim_{x\to\infty}\eta(x)=0$
\[
\frac{\phi(n+1)-\phi(n)}{\phi(n)} = \frac{\int_n^{n+1}\eta(u)du}{\phi(n)} \leq \frac{\eta(n)}{\phi(n)} \to 0 \mbox{ as }n\to\infty.
\]
Hence $\lim_{n\to\infty}\phi(n+1)/\phi(n) = 1$. Thus for any $x \in [n(x),n(x)+1)$ our previous arguments show that (suppressing $x$--dependence in $n$)
\[
\frac{f(x)}{\phi(x)} \leq \frac{f(n+1)}{\phi(n)} = \frac{f(n+1)}{\phi(n+1)}\frac{\phi(n+1)}{\phi(n)} \to 1 \mbox{ as } x\to\infty.
\]
Likewise 
\[
\frac{f(x)}{\phi(x)} \geq \frac{f(n)}{\phi(n+1)} = \frac{f(n)}{\phi(n)}\frac{\phi(n)}{\phi(n+1)} \to 1 \mbox{ as } x\to\infty.
\]
Note that $\lim_{x\to\infty}\phi'(x)=0$ since $\lim_{x\to\infty}\eta(x)=0$ by hypothesis and hence by Lemma \ref{to_zero} $\phi$ is sublinear. Therefore $f$ is also sublinear. \\\\*
We have chosen $h_n$ so that $h_n > \eta(n)$ for each $n$ and $f'(n+w_n/2)=h_n$. Hence
\[
\limsup_{x\to\infty}f'(x) \geq \lim_{n\to\infty}f'(n+\tfrac{w_n}{2}) = \lim_{n\to\infty}h(n) = L.
\]
Also $\lim_{x\to\infty}\eta(x)=0$ implies that $\liminf_{x\to\infty}f'(x)=0$.
\subsection{Example \ref{second_example}}
All of the arguments from Example \ref{first_example} also apply here with minor changes. In \eqref{f_n_ineq} we now have $T_n \to \bar{L} \in (0,\infty)$ and we can proceed as before. 

\bibliography{references}
\bibliographystyle{abbrv}
\end{document}